\documentclass[reqno]{amsart}

\usepackage{xcolor}
\usepackage{amssymb,amsmath,amsthm,amsfonts}
\usepackage{latexsym}
\usepackage{bbm}
\usepackage{mathrsfs}
\usepackage{hyperref}
\usepackage{makecell}
\usepackage{multirow}
\usepackage{diagbox}
\usepackage{braket}
\usepackage{enumitem}

\usepackage[a4paper,textwidth=158mm,textheight=232mm,centering,
            headheight=14pt,headsep=7mm,footskip=12mm]{geometry}
\allowdisplaybreaks[1]
\setlist[itemize]{leftmargin=2.4em,labelsep=.6em,
                 topsep=.5\baselineskip,itemsep=.3\baselineskip,parsep=0pt}
\setlist[enumerate]{leftmargin=2.4em,labelsep=.6em,
                   topsep=.5\baselineskip,itemsep=.4\baselineskip,parsep=0pt}
\hypersetup{hidelinks,pdftitle={Endpoint Mapping Properties of Wave Operators
for Schrodinger Operators in Dimensions n >= 3},
pdfauthor={Han Cheng, Changxing Miao, Xiaohua Yao}}

\makeatletter\renewcommand*{\eqref}[1]{%
\hyperref[{#1}]{\textup{\tagform@{\!\!\ref*{#1}}}}%
}\makeatother 

\makeatletter

\@addtoreset{equation}{section}
\makeatother
\theoremstyle{plain}
\newtheorem{theorem}{Theorem}[section]
\newtheorem{lemma}[theorem]{Lemma}
\newtheorem{proposition}[theorem]{Proposition}

\theoremstyle{definition}

\newtheorem{remark}[theorem]{Remark}

\newcommand{\les}{\lesssim}

\def\ac{\mathop{\mathrm{ac}}\nolimits}

\def\R{{\mathbb{R}}}

\def\C{{\mathbb{C}}}

\def\<{{\langle}}
\def\>{{\rangle}}

\makeatletter
\renewcommand{\l@section}{\@tocline{1}{0pt}{1em}{2.5pc}{}}
\renewcommand{\l@subsection}{\@tocline{2}{0pt}{2em}{2pc}{}}
\makeatother

\title[Endpoint mapping properties of wave operators]
{Endpoint Mapping Properties of Wave Operators for
Schr\"odinger Operators in Dimensions $n\ge3$}
\author{Han Cheng, Changxing Miao, Xiaohua Yao}
\address {Han Cheng, Institute of Applied Physics and Computational Mathematics, Beijing, 100088, People's Republic of China}
\email{chmathh@163.com}

\address {Changxing Miao, School of Mathematics and Physics,
University of Science and Technology Beijing,
Beijing 100083, China}
\email{miao\_changxing@ustb.edu.cn}

\address {Xiaohua Yao, School of Mathematics and Statistics, Key Laboratory of Nonlinear Analysis and Applications (Ministry of Education), Central China Normal University, Wuhan, 430079, P.R. China}
\email{yaoxiaohua@ccnu.edu.cn}
\date{\today}


\begin{document}
\begin{abstract}
We study endpoint mapping properties of low-energy wave operators for
Schr\"odinger operators $H=-\Delta+V$ on $\mathbb R^n$, $n\ge3$.
In dimension four, we prove that a zero-energy resonance prevents
$L^1$ boundedness, whether or not zero is also an eigenvalue, under
$|V(x)|\lesssim\langle x\rangle^{-\beta}$ with $\beta>10$.
For a zero-energy eigenvalue without a resonance, we obtain a complete
low-energy $L^p$ classification in every dimension $n\ge3$ under
$\beta>n+4$. In particular, $L^\infty$ boundedness is equivalent to
the vanishing of the zeroth, first, and harmonic second moments of
$V\psi$ for every zero-energy eigenfunction $\psi$.
The proof identifies the finite-rank obstruction and shows that the
remaining eigenvalue correction cannot cancel its critical asymptotic
profiles. The same conclusions hold for the full wave operators when
the corresponding high-energy bounds are available.
\end{abstract}

\maketitle

\section{Introduction}

\subsection{Background and problems}
Let $H_0=-\Delta$ and $H=H_0+V$ be the free and perturbed
Schr\"odinger operators on $\R^n$, where $V$ is a real-valued
potential. If $|V(x)|\les \langle x\rangle^{-\delta}$ for some
$\delta>1$, then the wave operators, defined as the strong limits in
$L^2(\R^n)$,
\[
W_\pm=W_\pm(H,H_0)
:=
\operatorname*{s-lim}_{t\to\pm\infty}
e^{itH}e^{-itH_0},
\]
exist and are asymptotically complete (see, e.g., Agmon \cite{Agmon} or
Reed and Simon \cite{Reed-Simon-III}); in particular
$W_\pm^{*}W_\pm=I$ and $W_\pm W_\pm^{*}=P_{\ac}(H)$, where $P_{\ac}(H)$
denotes the projection onto the absolutely continuous spectral subspace
of $H$. In this paper we study the mapping properties of $W_\pm$ on
$L^p(\R^n)$, with emphasis on the endpoints $p=1$ and $p=\infty$.

A principal reason for interest in such bounds is the intertwining identity
\begin{equation}\label{eq:intertwining}
m(H)P_{\ac}(H)=W_\pm m(H_0)W_\pm^*,
\end{equation}
valid for every bounded Borel function $m$. Once endpoint bounds for
$W_\pm$ are available, \eqref{eq:intertwining} transfers mapping
properties of the free multiplier $m(-\Delta)$ to the perturbed
operator; the $L^\infty$-bound of the wave operators, for instance, is
precisely what is needed to transfer the free $L^1$--$L^\infty$ dispersive estimate to $e^{itH}P_{\ac}(H)$. A second motivation is the
factorization $W_\pm=\mathcal F_\mp^{*}\mathcal F$ of the wave
operators through the distorted Fourier transform of $H$, which makes
the $L^p$ theory of wave operators a natural framework for the harmonic
analysis of $H$. We refer to the introduction of \cite{CMY-2D} for a
more detailed discussion of these two perspectives and for further
references.

The $L^p$ theory of wave operators was initiated by Yajima
\cite{Yajima-JMSJ-95, Yajima-1995-even} and has since been developed in
all spatial dimensions. The answer depends sensitively on the dimension
$n$ and on the spectral behavior of $H$ at zero energy, namely on the
presence of a zero-energy resonance or of zero-energy eigenfunctions;
see Subsection~\ref{known results} for a review. While the non-endpoint range $1<p<\infty$ is essentially understood in
every dimension, the endpoints are considerably more delicate. The
two-dimensional endpoint problem was solved in \cite{CMY-2D}, where the
following two endpoint problems were singled out
as remaining open:
\begin{itemize}
\item[($a$)] in $\R^4$, the $L^1$-unboundedness of $W_\pm$ when $H$
has a zero-energy resonance;

\item[($b$)] for $n\ge 3$, necessary and sufficient conditions for
the $L^\infty$-boundedness of $W_\pm$ when zero is an eigenvalue of
$H$ and no resonance is present.
\end{itemize}
We resolve these two problems for the low-energy wave operators.
The conclusions extend to the full wave operators whenever the
corresponding high-energy bounds are available.

We briefly place the two problems in context. In dimension four,
Yajima \cite{Yajima-2022-4} proved that, in the presence of a
zero-energy resonance, $W_\pm$ are bounded on $L^p(\R^4)$ for
$1<p\le2$ and unbounded for $2<p\le\infty$; whether they are bounded on
$L^1(\R^4)$ was left open. The question is already subtle at the level
of the comparison with other dimensions: in dimension three a resonance
destroys the $L^1$-bound \cite{Yajima-2016-3d}, whereas in dimensions
$n\ge5$ the wave operators may be bounded on $L^1$ even when zero is an
eigenvalue
\cite{Finco-Yajima-2006-even, Yajima-JMSJ-2006, Goldberg-Green-2016}.

Problem ($b$) concerns the case in which zero is an eigenvalue of $H$
but not a resonance. There the boundedness of $W_\pm$ on $L^p(\R^n)$ is
governed by the vanishing of the moments
\begin{equation}\label{eq:moments}
\int_{\R^n} x^\alpha V(x)\psi(x)\,dx
\end{equation}
of the zero-energy eigenfunctions $\psi$: roughly speaking, the larger
the exponent $p$, the more moments are required to vanish (see
Subsection~\ref{known results}). The endpoint $p=\infty$ requires a separate analysis.
Yajima \cite{Yajima-2016,Yajima-2016-3d} proved a sufficient
condition in dimension three, and Erdo\u{g}an, Green, and LaMaster
\cite{Erd-Gre-arxiv} obtained sufficient conditions in all dimensions
$n\ge3$. These results assume that \eqref{eq:moments} vanishes
for every $|\alpha|\le2$ and every zero-energy eigenfunction.
We identify the exact condition: in addition to the zeroth and
first moments, only the harmonic second moments must vanish.
This agrees with the role of harmonic second moments in the
two-dimensional endpoint analysis of \cite{CMY-2D}.

\subsection{Main results}
To state the results, we first recall the threshold terminology; for a
systematic discussion we refer to \cite{JN01} and the references
therein. A \emph{zero-energy eigenfunction} of $H$ is a nontrivial
$L^2$-solution of
\begin{equation}\label{eq:eigen}
H\psi=-\Delta\psi+V\psi=0
\end{equation}
in the sense of distributions. A \emph{zero-energy resonance} is a distributional solution of \eqref{eq:eigen}
belonging to $L^2_{-1/2-}(\mathbb{R}^3)\setminus L^2(\mathbb{R}^3)$ in dimension $n=3$,
and to $L^2_{-0-}(\mathbb{R}^4)\setminus L^2(\mathbb{R}^4)$ in dimension $n=4$.
Under the decay assumptions stated below, such resonances do not occur in dimensions $n\ge5$.

Choose an even function $\chi\in C_c^\infty(\mathbb R)$ with $0\le\chi\le1$ and define $\tilde\chi$ by
\begin{equation}\label{cutoff-fun}
\chi(\lambda)=1\quad(|\lambda|\le1/2),\qquad\chi(\lambda)=0\quad(|\lambda|\ge1),\qquad\tilde\chi=1-\chi.
\end{equation}
For a fixed cutoff scale $\lambda_0>0$, we write
\[
\chi_0(\lambda):=\chi(\lambda/\lambda_0),\qquad
\widetilde\chi_0(\lambda):=1-\chi_0(\lambda).
\]
This notation will be used throughout the paper.
We denote
\[
W_\pm^{\mathrm{low}}=W_\pm \circ\chi_0(\sqrt{-\Delta}),\quad  W_\pm^{\mathrm{high}}=W_\pm \circ \widetilde\chi_0(\sqrt{-\Delta}).
\]

Our first result resolves the low-energy obstruction in problem ($a$).
\begin{theorem}\label{thm:main-4d}
Let $H=-\Delta+V$ be a Schr\"odinger operator on $\R^4$ with
$|V(x)|\les\langle x\rangle^{-\beta}$ for some
$\beta>10$. If $H$ has a zero-energy resonance, then the wave
operators $W_\pm^{\mathrm{low}}$ are unbounded on $L^1(\R^4)$ for every sufficiently small cutoff scale $\lambda_0>0$.
\end{theorem}

Together with the known non-endpoint bounds and unboundedness results
of Yajima \cite{Yajima-2022-4}, Theorem~\ref{thm:main-4d} closes
the missing $L^1$ endpoint in the four-dimensional resonant case.
Whenever the high-energy component is bounded in the relevant spaces,
the full wave operators are bounded precisely for $1<p\le2$.
This complements the resonance results in dimensions one
\cite{Weder-99-CMP,Huang-Yao-2026}, two \cite{CMY-2D}, and three
\cite{Yajima-2016,Yajima-2016-3d}.

We next consider problem ($b$). The classification is expressed
in terms of three nested subspaces of the zero-energy eigenspace.
Let $\mathcal{E}=\ker H|_{L^2}$ and define $\mathcal{E}_0\supset\mathcal{E}_1\supset\mathcal{E}_2$ by
\begin{align*}
\mathcal{E}_0&=\Bigl\{\psi\in\mathcal{E}:\int_{\R^n} V\psi\,dx=0\Bigr\},\\
\mathcal{E}_1&=\Bigl\{\psi\in\mathcal{E}_0:\int_{\R^n}x_j V\psi\,dx=0,\ \forall j\Bigr\},\\
\mathcal{E}_2&=\Bigl\{\psi\in\mathcal{E}_1:\int_{\R^n}P(x)V(x)\psi(x)\,dx=0,\
\forall P\in\mathcal{A}_2\Bigr\},
\end{align*}
where $\mathcal{A}_2$ is the space of homogeneous harmonic polynomials of degree $2$ on $\R^n$.

\begin{theorem}\label{thm:main-infty}
Let $n\ge3$ and let $|V(x)|\lesssim\langle x\rangle^{-\beta}$, where
$\beta>n+4$. Fix a sufficiently small cutoff scale $\lambda_0>0$. Suppose zero is an eigenvalue of $H=-\Delta+V$ and
no zero-energy resonance is present. Then
\begin{enumerate}
\item[\emph{(i)}] $W_\pm^{\mathrm{low}}$ are bounded on
$L^p(\R^n)$ for $1\le p<\frac{n}{2}$ when $n\ge5$, and for
$1\le p\le2$ when $n=3, 4$; if
$\mathcal{E}\setminus\mathcal{E}_0\neq\emptyset$
(which is possible only when $n\ge5$), then they are unbounded on
$L^p(\R^n)$ for every
$\frac{n}{2}\le p\le\infty$;
\item[\emph{(ii)}] assume $\mathcal{E}=\mathcal{E}_0$; then
$W_\pm^{\mathrm{low}}$ are bounded on
$L^p(\R^n)$ for $1\le p<n$; if moreover
$\mathcal{E}_0\setminus\mathcal{E}_1\neq\emptyset$, they are unbounded
on $L^p(\R^n)$ for every $n\le p\le\infty$;
\item[\emph{(iii)}] assume $\mathcal{E}=\mathcal{E}_1$; then
$W_\pm^{\mathrm{low}}$ are bounded on
$L^p(\R^n)$ for $1\le p<\infty$; if moreover
$\mathcal{E}_1\setminus\mathcal{E}_2\neq\emptyset$, they are unbounded
on $L^\infty(\R^n)$;
\item[\emph{(iv)}] assume $\mathcal{E}=\mathcal{E}_2$; then
$W_\pm^{\mathrm{low}}$ are bounded on
$L^\infty(\R^n)$.
\end{enumerate}
\end{theorem}

We make some comments on Theorem~\ref{thm:main-infty}.
\begin{remark}
For $n=3,4$, under our assumptions on $V$, the identity $\mathcal{E}=\mathcal{E}_0$
holds automatically; see \cite{JK, J84, Yajima-1995-even}. Thus part~\textup{(ii)} gives boundedness for $1\le p<n$
in these dimensions, including the usual $L^2$ bound.
\end{remark}

\begin{remark}\label{rem:decay-and-endpoint-novelty}
In the presence of a zero-energy eigenvalue but no zero-energy
resonance, Theorem~\ref{thm:main-infty} establishes the low-energy
classification under the decay assumption $\beta>n+4$.
This is weaker than the assumption in \cite[Theorem~1.2]{Erd-Gre-arxiv},
which, for the classical Schr\"odinger operator, requires $\beta>n+6$
in odd dimensions and $\beta>n+5$ in even dimensions.
Moreover, to the best of our knowledge, this is the first
necessary and sufficient characterization of $L^\infty$-boundedness
in the eigenvalue case for every dimension $n\ge3$:
\[
W_\pm^{\mathrm{low}}\text{ are bounded on }L^\infty(\mathbb R^n)
\quad\Longleftrightarrow\quad
\mathcal E=\mathcal E_2.
\]
In particular, after the zeroth- and first-order moment conditions
are imposed, the endpoint is characterized precisely by the
vanishing of the harmonic second-order moments.
\end{remark}


\begin{remark}[Harmonic second moments]
The condition at $p=\infty$ is the higher-dimensional counterpart of
the one found in dimension two in \cite{CMY-2D}, where, in the
eigenvalue case, the $L^\infty$-endpoint was shown to depend on the
degree-two harmonic moments of $V\psi$, in addition to the presence
of an s-wave resonance and the absence of p-wave resonances. Thus, in
every dimension $n\ge2$, the $L^\infty$-endpoint of the wave
operators in the eigenvalue case is governed by second-order
threshold data of $H$.
\end{remark}

\begin{remark}
The statements here concern the low-energy wave operators.
They extend to the full wave operators whenever the corresponding
high-energy $L^p$ bounds are available under the chosen hypotheses
on $V$; see \cite{Yajima-JMSJ-95,Yajima-1995-even,Yajima-JMSJ-2006}. Indeed, it has been shown in \cite{Yajima-JMSJ-95,Yajima-1995-even,Yajima-JMSJ-2006} that $W_\pm^{\mathrm{high}}$ is bounded on $L^p$ for all $1\le p\le \infty$ and any fixed $\lambda_0$ if
\[
|V(x)|\les \langle x\rangle^{-n-2-}, \quad \mathcal{F}(\langle x\rangle^{2\sigma}V)\in L^{\frac{n-1}{n-2}} ~\text{ with $\sigma>\frac{n-2}{n-1}$}.
\]
Under these additional high-energy assumptions, Theorem~\ref{thm:main-infty}
therefore gives the same classification for the full wave operators
when zero is an eigenvalue and no zero-energy resonance is present.
\end{remark}

\subsection{Known results}\label{known results}
We briefly review the known results on the $L^p$-boundedness of Schr\"odinger wave operators.
A \emph{zero-energy resonance} is a distributional solution of \eqref{eq:eigen}
belonging to $L^2_{-1/2-}(\mathbb{R}^n)\setminus L^2(\mathbb{R}^n)$ in dimension $n=1, 3$,
and to $L^2_{-0-}(\mathbb{R}^4)\setminus L^2(\mathbb{R}^4)$ in dimension $n=4$.

\medskip
\noindent\textbf{Dimension $n=1$.}
Weder \cite{Weder-99-CMP} proved, under suitable decay assumptions on $V$, that the wave operators $W_\pm$ are bounded on $L^p(\mathbb R)$ for every $1<p<\infty$, both in the regular case and in the presence of a zero-energy resonance. He further showed that, if zero is resonant and the zero-energy Jost solution satisfies
\begin{equation}\label{eq:Jost-endpoint}
\lim_{x\to-\infty} f_1(x,0)=1,
\end{equation}
then $W_\pm$ are bounded on $L^1(\mathbb R)$ and $L^\infty(\mathbb R)$. Galtbayar and Yajima \cite{Gal-Yaj-00} and D'Ancona and Fanelli \cite{AF06} subsequently established $L^p$-boundedness for $1<p<\infty$ under alternative, and in part weaker, assumptions on $V$.

Weder's low-energy decomposition exhibits Hilbert-transform terms in the remaining cases, which suggests the failure of strong endpoint boundedness. The actual unboundedness on $L^1(\mathbb R)$ and $L^\infty(\mathbb R)$ in the regular case, as well as in the resonant case where \eqref{eq:Jost-endpoint} fails, was established only recently by Huang and Yao \cite{Huang-Yao-2026}. Together with Weder's positive endpoint result, this completes the classification of the $L^p$-boundedness of the one-dimensional wave operators under the assumptions considered there. We recall that, for potentials satisfying $|V(x)|\les\langle x\rangle^{-2-}$, zero cannot be an $L^2$-eigenvalue in dimension one; see, for example, \cite{Deift-T}.

\medskip
\noindent\textbf{Dimension $n=2$.}
When zero is a regular point of $H$, Yajima \cite{Yajima-99-CMP} and Jensen--Yajima \cite{Jensen-Yajima-02-CMP} proved that $W_\pm$ are bounded on $L^p(\mathbb R^2)$ for every $1<p<\infty$. Erdo\u{g}an, Goldberg, and Green \cite{Erdogan-Goldberg-Green-JFA-2018} showed that the same conclusion remains valid if zero is a first-kind threshold singularity, or a zero-energy eigenvalue with no threshold resonance present.

Yajima \cite{Yajima-2022} subsequently obtained a sharp classification in the non-endpoint range: $W_\pm$ are bounded on $L^p(\mathbb R^2)$ for all $1<p<\infty$ if and only if $H$ has no p-wave resonance. If a p-wave resonance is present, then $W_\pm$ are bounded for $1<p\le 2$ and unbounded for $2<p<\infty$. The endpoint cases $p=1$ and $p=\infty$, for all possible
zero-energy spectral configurations, were settled in \cite{CMY-2D}.

\medskip
\noindent\textbf{Dimension $n=3$.}
If zero is a regular point, Yajima \cite{Yajima-JMSJ-95} established the boundedness of $W_\pm$ on $L^p(\mathbb R^3)$ for the full range $1\le p\le\infty$. Beceanu and Schlag \cite{Beceanu-Schlag-20} later obtained the same conclusion under substantially weaker assumptions on the potential.

If a zero-energy resonance is present, Yajima \cite{Yajima-2016,Yajima-2016-3d} proved that $W_\pm$ are bounded on $L^p(\mathbb R^3)$ for $1<p<3$, but are unbounded on $L^1(\mathbb R^3)$ and on $L^p(\mathbb R^3)$ for $3\le p\le\infty$.

Suppose instead that zero is an eigenvalue but no zero-energy resonance is present. Then $W_\pm$ are bounded on $L^p(\mathbb R^3)$ for $1\le p<3$. Moreover, they are bounded for $3\le p<\infty$ if and only if every zero-energy eigenfunction $\psi$ satisfies
\begin{equation}\label{vanishing-condi}
\int_{\mathbb R^3} x^\alpha V(x)\psi(x)\,dx=0,
\qquad |\alpha|\le 1.
\end{equation}
If this condition fails, $W_\pm$ are unbounded on $L^p(\mathbb R^3)$ for $3\le p<\infty$. Yajima \cite{Yajima-2016-3d} and Erdo\u{g}an-Green-Lamaster \cite{Erd-Gre-arxiv} further proved that $W_\pm$ are bounded on $L^\infty(\mathbb R^3)$ provided that condition \eqref{vanishing-condi} holds for all $|\alpha|\le 2$ and for every zero-energy eigenfunction $\psi$.

\medskip
\noindent\textbf{Dimension $n=4$.}
When zero is a regular point, Yajima \cite{Yajima-1995-even} proved that $W_\pm$ are bounded on $L^p(\mathbb R^4)$ for all $1\le p\le\infty$. If a zero-energy resonance is present, Yajima \cite{Yajima-2022-4} showed that $W_\pm$ are bounded on $L^p(\mathbb R^4)$ for $1<p\le 2$ and unbounded for $2<p\le\infty$.

Suppose that zero is an eigenvalue but no resonance is present. The results of Jensen--Yajima \cite{Jensen-Yajima-08-PLMS} and Goldberg--Green \cite{Goldberg-Green-2017} imply that $W_\pm$ are bounded on $L^p(\mathbb R^4)$ for $1\le p<4$ in general. If every zero-energy eigenfunction $\psi$ satisfies the analogue of condition \eqref{vanishing-condi} in $\mathbb R^4$, then $W_\pm$ are bounded on $L^p(\mathbb R^4)$ for every $1\le p<\infty$. Conversely, Yajima \cite{Yajima-2022-4} proved that, if this first-moment condition fails, then $W_\pm$ are unbounded on $L^p(\mathbb R^4)$ for $4<p<\infty$.
Erdo\u{g}an-Green-Lamaster \cite{Erd-Gre-arxiv} obtained the sufficient condition for $L^\infty$-boundedness.

\medskip
\noindent\textbf{Dimensions $n\ge 5$.}
In these dimensions, non-$L^2$ zero-energy resonances do not occur. If zero is a regular point, the wave operators are bounded on $L^p(\mathbb R^n)$ throughout the full range $1\le p\le\infty$; see \cite{Finco-Yajima-2006-even,Yajima-JMSJ-95,Yajima-1995-even,Yajima-JMSJ-2006}.

If zero is an eigenvalue, Finco--Yajima \cite{Finco-Yajima-2006-even,Yajima-JMSJ-2006} and Goldberg--Green \cite{Goldberg-Green-2016} proved that $W_\pm$ are bounded on $L^p(\mathbb R^n)$ for $1\le p<n/2$ without additional cancellation assumptions. If every zero-energy eigenfunction $\psi$ satisfies
\[
\int_{\mathbb R^n} V(x)\psi(x)\,dx=0,
\]
then $W_\pm$ are bounded for $1\le p<n$. If, in addition,
\[
\int_{\mathbb R^n} x_j V(x)\psi(x)\,dx=0,
\qquad j=1,\ldots,n,
\]
then they are bounded for every $1\le p<\infty$. In \cite{Erd-Gre-arxiv}, the authors obtain $L^\infty$-boundedness under the condition
$\int_{\mathbb{R}^n} x^{\alpha} V(x)\psi(x)\,dx=0$ for all $|\alpha|\le 2$.  Yajima \cite{Yajima-2016} further showed that these moment conditions are also necessary in the corresponding upper ranges: the zeroth-moment condition is necessary for $n/2\le p<n$, while the zeroth- and first-moment conditions are necessary for $n\le p<\infty$.

\subsection{Outline of the argument}
The two endpoint questions have different obstructions. In dimension
four, the resonant block has a nonzero zeroth moment. Testing the
stationary formula on functions with uniformly bounded $L^1$ norms
and shrinking Fourier holes at zero produces a $\log\log k$
divergence; the eigenvalue terms do not contribute to that leading
coefficient.

In the pure eigenvalue case, the leading pole produces the finite-rank
operator $\mathcal P_e\chi_0(\sqrt{-\Delta})$, modulo a bounded
remainder in the upper $L^p$ range. Its input profiles have successive
monopole, dipole, and harmonic quadratic tails. Their integrability
thresholds give the critical exponents $n/2$, $n$, and $\infty$.
The full inverse expansion also contains a correction in the eigenspace.
We retain this correction and prove that it belongs to the critical
dual spaces, thereby ruling out cancellation of the leading
obstructions.

\subsection{Notation and organization}
We use the following conventions throughout the paper:
\begin{itemize}
\item $A\lesssim B$ means $A\le CB$ for some constant $C>0$
independent of the relevant variables; $f(x)=\mathcal{O}(g(x))$ means
$|f(x)|\le C|g(x)|$. The notation $a+$, respectively $a-$, denotes
$a+\varepsilon$, respectively $a-\varepsilon$, for some arbitrarily
small fixed $\varepsilon>0$, and
$\langle x\rangle=(1+|x|^2)^{1/2}$.
\item For $1\le p\le\infty$, $L^p=L^p(\R^n;\C)$; for
$\sigma\in\R$, $L^2_\sigma(\R^n)=\{f:\langle x\rangle^\sigma
f\in L^2(\R^n)\}$; $L^{p,q}(\R^n)$ denotes the Lorentz space, and
$L^{1,\infty}(\R^n)$ the weak $L^1$ space.
\item For Banach spaces $X$ and $Y$, $\mathbb B(X,Y)$ denotes the
space of bounded linear operators from $X$ to $Y$, and
$\mathbb B(X)=\mathbb B(X,X)$.
\item The Fourier transform and its inverse are normalized as
\[
\widehat f(\xi)=\int_{\R^n}e^{-ix\cdot\xi}f(x)\,dx,
\qquad
\check g(x)=\frac{1}{(2\pi)^n}\int_{\R^n}e^{ix\cdot\xi}g(\xi)\,d\xi.
\]
\item The $L^2$ inner product is linear in its first argument:
\(\langle f,g\rangle=\int f\overline g\).
We use $\mathcal P_e$ for the orthogonal projection onto $\ker H$
in $L^2$; the auxiliary projections $S_1,S_2$ are distinct from it.

\item We denote by
\[
f^{(k)}(z)=\frac{d^k}{dz^k}f(z)
\]
the $k$-th derivative of a function $f$. We write
$f(\lambda)=\mathcal{O}_m(g(\lambda))$ for
$\lambda\in\Omega\subset\mathbb{R}$ if
\begin{equation*}
	\left|f^{(k)}(\lambda)\right|\lesssim_k |g(\lambda)|\,|\lambda|^{-k},
	\qquad \lambda\in\Omega,\quad k=0,1,\dots,m.
\end{equation*}
The notation extends naturally to $m=\infty$; for brevity, we also write
$\mathcal{O}(g(\lambda))=\mathcal{O}_0(g(\lambda))$.
\end{itemize}

The argument is organized as follows. Section~\ref{sec:pre} fixes
the stationary representation and free-resolvent estimates.
Section~\ref{sec:resonant} treats the four-dimensional resonant case.
Section~\ref{sec:eigenvalue} proves Theorem~\ref{thm:main-infty}.
We first bound the regular contribution, then analyze the leading
eigenvalue term, and finally show that the remaining correction
does not change the boundedness criterion.
Section~\ref{sec:exp of lemma} supplies the block-matrix inversion.

\section{Preliminaries}\label{sec:pre}

By the symmetry $\overline{W_+f}=W_-\overline{f}$, it suffices to study
$W_- \circ \chi_0(\sqrt{-\Delta})$. Our proof begins with the
stationary representation of the wave operators (see, e.g., Kuroda
\cite{Kuroda}):
\begin{equation}\label{eq-stationary formula}
\begin{aligned}
W_- \circ \chi_0(\sqrt{-\Delta})f
&=
\chi_0(\sqrt{-\Delta})f-\frac{1}{\pi i}
\int_0^\infty
\lambda R_V^+(\lambda^2)V
\bigl(R_0^+(\lambda^2)-R_0^-(\lambda^2)\bigr)f
\,\chi_0(\lambda)d\lambda  \\
&=:\chi_0(\sqrt{-\Delta})f-\mathcal W^{\mathrm{low}}f .
\end{aligned}
\end{equation}
Here $R_V^\pm(\lambda^2)$ denotes the boundary resolvent
$(H-(\lambda^2\pm i0))^{-1}$ and $R_0^\pm(\lambda^2)$ denotes the corresponding boundary resolvent of $(-\Delta-(\lambda^2\pm i0))^{-1}$.
Since $\chi_0(\sqrt{-\Delta})$ is a smooth compactly supported
Fourier multiplier, it is bounded on $L^p$ for every $1\le p\le\infty$.
It therefore remains to analyze $\mathcal W^{\mathrm{low}}$. Set
\begin{equation}\label{eq:def-of-M}
M^{+}(\lambda)=U+vR_0^+(\lambda^2)v,
\qquad v=\sqrt{|V|},\qquad w=Uv,
\end{equation}
where $U$ is multiplication by
\[
U(x)=\begin{cases}1,&V(x)\ge0,\\-1,&V(x)<0.\end{cases}
\]
Thus $U^2=I$, $V=Uv^2$, and $vw=wv=V$. The symmetric resolvent identity yields
\[
R_V^+(\lambda^2)V
=R_0^+(\lambda^2)vM^{+}(\lambda)^{-1}v,
\]
which gives
\begin{equation}\label{eq:rep-for-low}
\mathcal W^{\mathrm{low}}f=\frac{1}{\pi i}
\int_0^\infty
\lambda R_0^+(\lambda^2)vM^{+}(\lambda)^{-1}v
\bigl(R_0^+(\lambda^2)-R_0^-(\lambda^2)\bigr)f
\,\chi_0(\lambda)d\lambda.
\end{equation}

\subsection{Free resolvents and symbol estimates}\label{sec:free-r}
Let $r=|x-y|$, $z=\lambda r$, and $\nu=n/2-1$.
The free-resolvent kernels are
\[
R_0^\pm(\lambda^2)(x,y)
=\frac{\pm i}{4}\left(\frac{\lambda}{2\pi r}\right)^\nu
H_\nu^{(\pm)}(z),
\]
where $H_\nu^{(+)}$ and $H_\nu^{(-)}$ are the Hankel functions
of the first and second kind, respectively. In particular,
\[
R_0^\pm(\lambda^2)(x,y)=\frac{e^{\pm i\lambda r}}{4\pi r}
\qquad(n=3).
\]
The small-argument expansions of the Hankel functions
\cite{Abra,GR02} give, for $0<z\lesssim1$,
\begin{equation}\label{eq:kernel-est-1}
\begin{aligned}
R_0^\pm(\lambda^2)(x,y)
&=\lambda^{n-2}\left[
\sum_{j=0}^{(n-3)/2}b_jz^{2j+2-n}
+a_0^\pm+\Gamma_0^\pm(z)\right], && n\text{ odd},\\
R_0^\pm(\lambda^2)(x,y)
&=\lambda^{n-2}\left[
\sum_{j=0}^{n/2-2}b_jz^{2j+2-n}
+b\log z+a_0^\pm+\Gamma_0^\pm(z)\right], && n\text{ even}.
\end{aligned}
\end{equation}
Here all coefficients depend on $n$, the $b_j$ and $b$ are real,
and $a_0^- =\overline{a_0^+}$. The remainders satisfy
\[
\Gamma_0^\pm(z)=
\begin{cases}
\mathcal O_\infty(z),&n\text{ odd},\\
\mathcal O_\infty\bigl(z^2(1+|\log z|)\bigr),&n\text{ even}.
\end{cases}
\]
The kernel $b_j|x-y|^{2j+2-n}$ is that of
$(-\Delta)^{-j-1}$ for the indices in the finite sums.
In particular, the even-dimensional expansion contains a logarithmic
term, which will be included in the remainder when we estimate
the small-argument contribution below.

With the cutoff functions introduced in \eqref{cutoff-fun}, one has
\begin{align}\label{eq:kernel-est-2}
&R_0^\pm(\lambda^2)(x,y)\chi(z)=\frac{e^{\pm iz}}{|x-y|^{n-2}}\omega_{0}^{\pm}(z),
\\
\label{eq:kernel-est-3}
&R_0^\pm(\lambda^2)(x,y)\tilde\chi(z)=\frac{e^{\pm iz}}{|x-y|^{n-2}} (z)^{\frac{n-3}{2}}\omega_1^{\pm}(z),
\\
\label{eq:kernel-est-4}
&R_0^+(\lambda^2)(x,y)-R_0^-(\lambda^2)(x,y)=\frac{e^{iz}}{|x-y|^{\frac{n-1}{2}}} \lambda^{\frac{n-3}{2}}\omega_+(z)+\frac{e^{-iz}}{|x-y|^{\frac{n-1}{2}}}\lambda^{\frac{n-3}{2}}\omega_-(z),
\end{align}
with
\begin{equation}\label{eq:est-for-omega}
\omega_0^\pm(z),\ \omega_1^\pm(z)=\mathcal{O}_{\infty}(1), \quad  \omega_+(z)=-\overline{\omega_-}(z)=\mathcal{O}_{\infty}(1).
\end{equation}
Further details on \eqref{eq:kernel-est-1}--\eqref{eq:est-for-omega} can be found in \cite{Goldberg-Green-dis-13, Goldberg-Green-dis-17, J84, Yajima-1995-even}.

Fix a sufficiently small $\lambda_0>0$ and let $k_0=[(n+3)/2]$.
We write $E(\lambda)=\widetilde{\mathcal O}(\lambda^\alpha)$ if
\[
\sup_{0\le k\le k_0}\sup_{0<\lambda<\lambda_0}
\lambda^{k-\alpha}|\partial_\lambda^kE(\lambda)(x,y)|
\]
is the kernel of a bounded operator on $L^2$. An operator is
\emph{absolutely bounded} if the absolute value of its kernel
defines a bounded operator on $L^2$. Bounded multiplication
operators are included through their diagonal kernels.

Assume $|V(x)|\lesssim\langle x\rangle^{-\beta}$ with $\beta>n+4$.
Set $G_0=(-\Delta)^{-1}$,
$P f=\|V\|_1^{-1}v\langle f,v\rangle$, and $T=U+vG_0v$.
Then the expansions in \eqref{eq:kernel-est-1}  give
\begin{equation}\label{eq:exp of M}
\begin{aligned}
M^+(\lambda)&=T+\lambda a_0^+P+\lambda^2vG_1v+r(\lambda),&&n=3,\\
M^+(\lambda)&=T+\lambda^2(a_0^++b\log\lambda)P
+\lambda^2vG_1v+r(\lambda),&&n=4,\\
M^+(\lambda)&=T+\lambda^2vG_1v+r(\lambda),&&n\ge5,
\end{aligned}
\end{equation}
where the constants multiplying $P$ include its normalization and
\[
G_1(x,y)=\begin{cases}
\displaystyle\frac{|x-y|^{4-n}}{2(n-4)(n-2)|\mathbb S^{n-1}|},&n\ne4,\\[4pt]
\displaystyle-\frac1{8\pi^2}\log|x-y|,&n=4.
\end{cases}
\]
The remainder satisfies $r(\lambda)=\widetilde{\mathcal O}(\lambda^3)$.
To verify the absolute symbol estimate, write the kernel of
$r(\lambda)$ as $v(x)e(\lambda,|x-y|)v(y)$.
The small-argument expansion and the large-argument symbol estimates
give, for $0\le k\le k_0$,
\[
|\partial_\lambda^k e(\lambda,r)|
\lesssim\lambda^{3-k}\left(\langle r\rangle^2+
\mathbf1_{\{r<1\}}r^{\min\{4-n,0\}}\right).
\]
The local term is integrable in the difference variable, so its
sandwich by the bounded function $v$ is bounded on $L^2$ by
Young's inequality. The remaining term is dominated by
$C v(x)\langle x\rangle^2v(y)\langle y\rangle^2$ and is
Hilbert--Schmidt, since $\beta>n+4$ implies
$\langle x\rangle^2v\in L^2$. These two majorants establish
the stated absolute symbol estimate for all the required derivatives.

For an operator $\mathcal{N}$, we write $\mathcal{N}(x,y)$ for its integral kernel. Such a
kernel is called \emph{admissible} if
\begin{equation*}
\sup_{x\in\mathbb{R}^n}\int_{\mathbb{R}^n}|\mathcal{N}(x,y)|\,dy
+\sup_{y\in\mathbb{R}^n}\int_{\mathbb{R}^n}|\mathcal{N}(x,y)|\,dx<\infty.
\end{equation*}
Schur's test shows that the operator with such kernel is bounded on $L^p(\mathbb{R}^n)$ for every $1\le p\le\infty$. In particular, $\sup_{y\in\mathbb{R}^n}\int_{\mathbb{R}^n}|\mathcal{N}(x,y)|\,dx<\infty$ implies $L^1$-boundedness of $\mathcal N$ and $\sup_{x\in\mathbb{R}^n}\int_{\mathbb{R}^n}|\mathcal{N}(x,y)|\,dy$ implies the $L^\infty$-boundedness of $\mathcal N$.

\subsection{Technical lemmas}
We use the following two estimates from \cite[Lemmas~2.1 and 2.4]{CMY-2D}.
\begin{lemma}\label{lem-oscillatory estimate}
Let $0<\lambda_0<1$, and suppose that
$m(\lambda)= \mathcal{O}_{[b]+2}\big(|\log\lambda|^{-a} \lambda^b\big)$
for $\lambda\in (0,\lambda_0)$ and some fixed $a\ge 0$ and $b>-1$. Then
\begin{equation}\label{eq.oscillatory-est}
\left|\int_0^\infty e^{i\lambda x} m(\lambda) \chi_0(\lambda)
\, d\lambda\right| \lesssim_{m,\lambda_0}
\frac{1}{\langle x\rangle^{b+1}\log^{a}(2+|x|)}, \quad x \in \mathbb{R},
\end{equation}
where the cutoff function $\chi$ is given in \eqref{cutoff-fun}.
\end{lemma}

\begin{lemma}\label{lem-Fou-multiplier}
Let $a>0$ and $0<\lambda_0<1$ be fixed. Then
\[
\left|\mathcal{F}^{-1}\!\left(\log^{-a}|\cdot|
\chi_0\!\left(|\cdot|\right)\right)(x)
\right|\les_{a,n,\lambda_0}\frac{1}{|x|^n\log^{a+1}|x|},\quad |x|\gg 1.
\]
Accordingly, the Fourier multiplier operator
$\log^{-a}(\sqrt{-\Delta})\, \chi_0\big(\sqrt{-\Delta}\big)$ is
bounded on $\mathcal{H}^1(\mathbb{R}^n)$, $\mathrm{BMO}(\mathbb{R}^n)$,
and $L^p(\mathbb{R}^n)$ for every $1\le p\le \infty$.
\end{lemma}

The following is from \cite[Lemma 3.8]{GV06}.
\begin{lemma}\label{lemma-GV}
Let $n\ge 1$, $k<n$ and $k+\ell>n$. Then
\begin{equation*}
\int_{\mathbb{R}^n}|x-y|^{-k}\langle y\rangle^{-\ell}\,dy
\les_{n,k,\ell}  \langle x\rangle^{-\min\{k,\,k+\ell-n\}}.
\end{equation*}
\end{lemma}

\section{The four-dimensional resonant case}\label{sec:resonant}

This section is devoted to the proof of Theorem~\ref{thm:main-4d}. In view of the representation~\eqref{eq-stationary formula}, it suffices to show that $\mathcal{W}^{\mathrm{low}}$ is unbounded on $L^{1}(\R^{4})$ whenever $H$ has a resonance at zero energy.

To study the $L^{1}$-unboundedness of $\mathcal{W}^{\mathrm{low}}$, we recall the small-energy asymptotic expansion of $M^{+}(\lambda)^{-1}$ in four dimensions established in \cite[Propositions~3.4 and 5.3]{Er-Gold-G-13}. We follow the notation of~\cite{Er-Gold-G-13}: $G_{0}=(-\Delta)^{-1}$; $G_{1}$ is the operator with integral kernel
\[
G_{1}(x,y)=b_1\log|x-y|=-\frac{1}{8\pi^{2}}\log|x-y|;
\]
$P$ is the rank-one projection $Pf=\|V\|_{L^{1}}^{-1}v\langle f,v\rangle$; and $T=U+vG_{0}v$. Let $S_{1}$ be the Riesz projection onto the kernel of $T$ on $L^{2}(\R^{4})$, and let $S_{2}$ be the Riesz projection onto the kernel of $S_{1}PS_{1}$ on $S_{1}L^{2}(\R^{4})$. \emph{It has been shown in Section 7 of \cite{Er-Gold-G-13} that $S_1-S_2$ is not zero if there is a resonance of $H$ and $S_2=0$ when zero is not an eigenvalue of $H$}. Define $D_{2}=(S_{2}vG_{1}vS_{2})^{-1}$ as an operator on $S_{2}L^{2}(\R^{4})$, and note that $D_{2}=D_{2}S_{2}=S_2D_2$.
Finally, set $\Gamma_{3}=S_{1}vG_{2}vS_{1}$, where $G_{2}$ is the operator with kernel $c_{2}|x-y|^{2}$ in~\cite{Er-Gold-G-13}.

For constants $a,a_{1},a_{2}\in\R\setminus\{0\}$ and $z,z_{1},z_{2}\in\C\setminus\R$, we write
\begin{align*}
f^{\pm}(\lambda)&=\frac{1}{\lambda^{2}}\,\frac{1}{a\log\lambda+z+\mathcal{O}_{1}(\lambda^{0+})},
\\
g_{1}^{\pm}(\lambda)&=\lambda^{2}\bigl(a_{1}\log\lambda+z_{1}\bigr),
\\
g_{2}^{\pm}(\lambda)&=\lambda^{4}\bigl(a_{2}\log\lambda+z_{2}\bigr).
\end{align*}
(The minus boundary values are the complex conjugates of the plus boundary values. Only the plus sign is used below.)

The coefficient $f_1^+$ in the mixed resonance--eigenvalue case
has the same leading form as $f^+$, with its own constants:
$f_1^+(\lambda)=\lambda^{-2}(a\log\lambda+z+
\mathcal O_1(\lambda^{0+}))^{-1}$, $a\ne0$.

\begin{lemma}\label{lem:exp-of-in-M}
Assume that $|V(x)|\lesssim \langle x\rangle^{-\beta}$ for some $\beta>10$, and let $0<\lambda<\lambda_{0}$ with $\lambda_{0}>0$ fixed and sufficiently small.
\begin{enumerate}[label=\textup{(\roman*)},leftmargin=2.5em,itemsep=0.6ex]
\item If zero is a resonance but not an eigenvalue of $H$, then
\[
M^{+}(\lambda)^{-1}=f^{+}(\lambda)(S_{1}-S_2)+K+\mathcal{L}(\lambda).
\]
\item If zero is both a resonance and an eigenvalue of $H$, then
\[
M^{+}(\lambda)^{-1}=f_{1}^{+}(\lambda)S_{1}SS_{1}+\frac{S_{2}D_{2}S_{2}}{\lambda^{2}}+\frac{g_{2}^{+}(\lambda)}{\lambda^{4}}S_2D_{2}\Gamma_{3}D_{2}S_2+K+\mathcal{L}(\lambda).
\]
\end{enumerate}
In the two expansions above, $S$ is given by
\[
S=S_1-S_2+(S_1-S_2)B_1S_2+S_2B_2(S_1-S_2)+S_2B_3S_2, \quad  B_1, B_2, B_3 \in \mathbb{B}(L^2(\mathbb{R}^4)),
\]
and the absolute values of the kernels of $K$ and $\mathcal{L}(\lambda)$ define $L^2$-bounded operators. Moreover, $\mathcal{L}(\lambda)=\widetilde{\mathcal O}(|\log\lambda|^{-1})$.
\end{lemma}

\begin{remark}
The endpoint argument uses Lemma \ref{lem:exp-of-in-M} through order three,
rather than only the first derivative used in \cite{Er-Gold-G-13}. Although \cite{Er-Gold-G-13} states the remainder estimates
only for $k=0,1$, the same argument readily yields the bounds
for $k=2,3$ under the decay assumption imposed here.
\end{remark}

\begin{remark}
Choose $\lambda_0$ sufficiently small that both inverse expansions
hold throughout the support of $\chi_0$. We use this same cutoff
in the stationary formula and in Theorem~\ref{thm:main-4d}.
No extension of the asymptotic expansion to arbitrary positive
energies is needed.
\end{remark}

\subsection{The regular contribution}
We first show that $K+\mathcal L(\lambda)$ contributes a bounded
operator on $L^1$. We then prove that the remaining singular terms
give an unbounded operator. We present the case in which zero is
both a resonance and an eigenvalue; the case without a zero
eigenvalue follows by omitting the terms involving $S_2$.

\begin{lemma}\label{lem:L1-bound-of-K}
Let $K$ and $\mathcal L(\lambda)$ satisfy the absolute kernel estimates
in Lemma~\ref{lem:exp-of-in-M}. Assume that $|V(x)|\lesssim \langle x\rangle^{-\beta}$ for some $\beta>4$. Then the operator
\begin{equation}\label{eq:rem-part}
\frac{1}{\pi i}\int_{0}^{\infty}\lambda R_{0}^{+}(\lambda^{2})v(K+\mathcal{L}(\lambda))v
\bigl(R_{0}^{+}(\lambda^{2})-R_{0}^{-}(\lambda^{2})\bigr)\,\chi_0(\lambda)\,d\lambda
\end{equation}
is bounded on $L^1(\mathbb{R}^4)$.
\end{lemma}

\begin{proof}
To prove this lemma, we use the identities \eqref{eq:kernel-est-2}--\eqref{eq:kernel-est-4}.
The kernel of \eqref{eq:rem-part} can be written as a linear combination of terms of the form
\begin{equation}\label{eq:ke-of-W0}
\int_{0}^{\infty}\!\!\iint_{\mathbb{R}^8} e^{i\lambda\rho_{\pm}}
\mathcal{G}_{j}^{\pm}(\lambda, x, z, w, y)\, v(w)(K+\mathcal{L}(\lambda))(z, w) v(z)\, dz\, dw\,
\lambda\chi_0(\lambda)\,d\lambda
\end{equation}
for $j=0,1$, where $\rho_{\pm}=|x-z|\pm |w-y|$ and
\[
\mathcal{G}_{j}^{\pm}(\lambda,x, z, w, y)
=\frac{1}{|x-z|^{2-j/2}}\,\frac{1}{|w-y|^{3/2}}\,
\lambda^{\frac{1+j}{2}}\omega_{j}(\lambda|x-z|)\,\omega_{\pm}(\lambda|w-y|).
\]
By the estimates \eqref{eq:est-for-omega}, for every integer $k\geq0$,
\[
\bigl|\partial_{\lambda}^{k}\bigl(\omega_{j}(\lambda|x-z|)\,\omega_{\pm}(\lambda|w-y|)\bigr)\bigr|
\lesssim\lambda^{-k}
\]
uniformly in $x,y,z,w\in\mathbb{R}^{4}$.

Set
$
N(z,w)=\sup_{0\le k\le3}\sup_{0<\lambda<\lambda_0}\bigl|\lambda^k(\log\lambda)\,\partial_{\lambda}^k\mathcal{L}(\lambda)(z,w)\bigr|.
$
By the absolute symbol estimate for $\mathcal L$, this common
majorant defines a bounded operator on $L^2(\mathbb R^4)$.
Fix $0<\varepsilon<1/2$. Including the additional factor $\lambda$
in \eqref{eq:ke-of-W0}, apply Lemma~\ref{lem-oscillatory estimate}
with $b=(3+j)/2-\varepsilon$. This requires at most three derivatives
and gives
\begin{align*}
&\biggl|\int_{0}^{\infty}e^{i\lambda\rho_{\pm}}\lambda^{\frac{3+j}{2}}\omega_{j}(\lambda|x-z|)\omega_{\pm}(\lambda|w-y|) (K+\mathcal{L}(\lambda))(z, w)\,\chi_0(\lambda)\,d\lambda\biggr|
\\
&\qquad\qquad \lesssim \bigl(|K(z, w)|+|N(z,w)|\bigr)\langle\rho_{\pm}\rangle^{-\frac{5+j}{2}+\varepsilon}.
\end{align*}

Consequently, the $L^{1}$-norm of \eqref{eq:ke-of-W0} in the variable $x$ is bounded by
\begin{align*}
&\int_{\mathbb{R}^{4}}\!\!\int_{\mathbb{R}^{4}} \int_{\mathbb{R}^{4}}\frac{|v(z)v(w)|\bigl(|K(z, w)|+|N(z,w)|\bigr)}{|x-z|^{2-j/2}\,\langle\rho_{\pm}\rangle^{(5+j)/2-\varepsilon}\,|w-y|^{3/2}}\,dz\, dw\,dx
\\
&\quad \lesssim\int_{\mathbb{R}^{4}}\int_{\mathbb{R}^{4}}\frac{|v(z)v(w)|\bigl(|K(z, w)|+|N(z,w)|\bigr)}{|w-y|^{3/2}}\int_{0}^{\infty}\frac{r^{1+j/2}}{\langle r\pm|w-y|\rangle^{(5+j)/2-\varepsilon}}\,dr\,dz\, dw
\\
&\quad \lesssim\int_{\mathbb{R}^{4}}\int_{\mathbb{R}^{4}}|v(z)v(w)|\bigl(|K(z,w)|+N(z,w)\bigr)\bigl(|w-y|^{(j-1)/2}+|w-y|^{-3/2}\bigr)\, dz\, dw
\\
&\quad\lesssim \|v\|_{L^2}\left(\int_{\mathbb{R}^{4}}\langle w\rangle^{-\beta}\bigl(|w-y|^{j-1}+|w-y|^{-3}\bigr)\,dw\right)^{1/2}\lesssim 1,
\end{align*}
uniformly in $y\in\mathbb R^4$. Here polar coordinates give the
radial integral, which is bounded by $C(1+|w-y|^{1+j/2})$.
Cauchy--Schwarz and the $L^2$ bounds for $|K|$ and $N$ then give
the last displayed expression. Its uniform bound follows from
Lemma~\ref{lemma-GV} and the decay of $V$.
The $L^{1}$-boundedness of \eqref{eq:rem-part} now follows from Schur's test.
\end{proof}

\subsection{The resonant obstruction}
At this point, the proof of Theorem~\ref{thm:main-4d} will be complete once we show that
the contribution from the leading terms in part~\textup{(ii)} of Lemma~\ref{lem:exp-of-in-M}
is unbounded on $L^1$. This is the content of the following lemma.

\begin{lemma}\label{lem:unbounded-leading}
Assume that $|V(x)|\lesssim \langle x\rangle^{-\beta}$ for some $\beta>8$. Then the operator
$\mathcal{W}_{m}^{\mathrm{low}}$ defined by
\begin{equation}\label{eq:main-part}
\begin{aligned}
\frac{1}{\pi i}\int_{0}^{\infty}\!\lambda R_{0}^{+}(\lambda^{2})v
\Bigl(f_{1}^{+}(\lambda)S_{1}SS_{1}&+\frac{S_{2}D_{2}S_{2}}{\lambda^{2}}
+\frac{g_{2}^{+}(\lambda)}{\lambda^{4}}S_2D_{2}\Gamma_{3}D_{2}S_2\Bigr)v
\bigl(R_{0}^{+}(\lambda^{2})-R_{0}^{-}(\lambda^{2})\bigr)
\chi_0(\lambda)\,d\lambda
\end{aligned}
\end{equation}
is unbounded on $L^1(\mathbb{R}^4)$.
\end{lemma}

\begin{proof}
By the definition of $S_1$ and $S_2$, compactness of $vG_0v$,
and the Riesz--Schauder theory, the projections $S_1$ and $S_2$ are finite-rank operators.
Furthermore, $S_1-S_2$ is a rank-one projection onto the one-dimensional space spanned by
$(S_1-S_2)v$ whenever a resonance is present. Recall that
$
S=S_1-S_2+(S_1-S_2)B_1S_2+S_2B_2(S_1-S_2)+S_2B_3S_2.
$
Set $\psi_0=(S_1-S_2)v/\|(S_1-S_2)v\|_{L^2}$ and let $\{\psi_j\}_{j=1}^{m}$ be an orthonormal
basis of $S_2L^2(\mathbb{R}^4)$. In view of the above, we can represent the operator
$
f_{1}^{+}(\lambda)S_{1}SS_{1}+\frac{S_{2}D_{2}S_{2}}{\lambda^{2}}
+\frac{g_{2}^{+}(\lambda)}{\lambda^{4}}S_2D_{2}\Gamma_{3}D_{2}S_2
$
as
\begin{equation}\label{eq:exp-for-Op}
f_{1}^{+}(\lambda)\sum_{i,j=0}^{m}\alpha_{i,j}^{(1)}\psi_i\langle\cdot,\psi_j\rangle
+\frac{1}{\lambda^2}\sum_{i,j=1}^{m}\alpha_{i,j}^{(2)}\psi_i\langle\cdot,\psi_j\rangle
+\frac{g_{2}^{+}(\lambda)}{\lambda^{4}}\sum_{i,j=1}^{m}\alpha_{i,j}^{(3)}\psi_i\langle\cdot,\psi_j\rangle,
\end{equation}
where $\alpha_{0,0}^{(1)}=1$ and the parameters $\alpha_{i,j}^{(u)}$ for $i,j=0,1,\ldots,m$ and $u=1,2,3$ are fixed constants.

Note that
\begin{equation}\label{eq:Four_trans_on_S}
\mathcal{F}(R_0^+(\lambda^2)g)(\xi)=\frac{\widehat{g}(\xi)}{|\xi|^2-\lambda^2-i0}, \quad 
\frac{\lambda}{\pi i}\bigl(R_0^+(\lambda^{2})-R_0^-(\lambda^{2})\bigr)h(z)
=\bigl(\widehat{h}\,d\sigma_\lambda\bigr)^{\vee}(z),
\end{equation}
where $d\sigma_\lambda$ denotes the induced measure on the sphere $\lambda\mathbb{S}^3$.
Then, applying Plancherel's theorem together with \eqref{eq:exp-for-Op}, we obtain for any Schwartz
function $h$ that
\begin{equation}\label{eq:F-repre-of-Wm}\begin{aligned}
\mathcal{F}(\mathcal{W}_{m}^{\mathrm{low}}h)(\xi)
=&\sum_{i,j=0}^{m}\alpha_{i,j}^{(1)}\frac{\widehat{v\psi_i}(\xi)}{(2\pi)^4}
\lim_{\varepsilon\to 0^+}\int_{\mathbb{R}^4}
\frac{\overline{\widehat{v\psi_j}(\eta)}\,\widehat{h}(\eta)\chi_0(|\eta|)}{|\xi|^2-|\eta|^2-i\varepsilon}\,
f_{1}^{+}(|\eta|)\,d\eta \\
&+\sum_{i,j=1}^{m}\frac{\widehat{v\psi_i}(\xi)}{(2\pi)^4}
\lim_{\varepsilon\to 0^+}\int_{\mathbb{R}^4}
\frac{\overline{\widehat{v\psi_j}(\eta)}\,\widehat{h}(\eta)\chi_0(|\eta|)}{|\xi|^2-|\eta|^2-i\varepsilon}\,
\Bigl[\frac{\alpha_{i,j}^{(2)}}{|\eta|^2}+\frac{\alpha_{i,j}^{(3)}\,g_{2}^{+}(|\eta|)}{|\eta|^4}\Bigr]
d\eta.
\end{aligned}
\end{equation}

Since $S_2v=0$, we have $\int_{\mathbb{R}^4}(\psi_j v)(x)\,dx=0$ for all $j=1,\dots,m$.
Noting that $|v(x)|\lesssim\langle x\rangle^{-4-}$ and $\psi_j\in L^2(\mathbb{R}^4)$, we obtain
$\langle x\rangle(\psi_j v)(x)\in L^1(\mathbb{R}^4)$, and hence
$\widehat{v\psi_j}(\xi)=\mathcal{O}(|\xi|)$ for $j=1,\dots,m$.  Moreover,
\begin{equation}\label{eq:nonzero-vpsi0}
\widehat{v\psi_0}(0)=\int_{\mathbb{R}^4}v(x)\psi_0(x)\,dx
=\frac{\langle v,(S_1-S_2)v\rangle}{\|(S_1-S_2)v\|_{L^2}}
=\|(S_1-S_2)v\|_{L^2}\neq0,
\end{equation}

Suppose, for a contradiction, that $\mathcal W_m^{\mathrm{low}}$ is $L^1$ bounded. For every Schwartz function $h$, the Fourier transform of
$\mathcal W_m^{\mathrm{low}}h$ is then continuous at the origin. Now let $\{h_k\}_{k\ge1}$ be the sequence of Schwartz functions defined by
$
\widehat{h_k}(\xi)=\chi_0(2|\xi|)\bigl(1-\chi(k|\xi|)\bigr).
$
The inverse Fourier transforms of the two smooth cutoffs have uniformly bounded $L^1$ norms under dilation; hence $\sup_k\|h_k\|_{L^1}\lesssim 1$. Since
$\widehat{v\psi_j}(\xi)=\mathcal{O}(|\xi|)$ for $j=1,\dots,m$, the representation
\eqref{eq:F-repre-of-Wm} can be evaluated at $\xi=0$ for each fixed
$k$. Indeed, $\widehat h_k$ vanishes in a neighborhood of zero,
so the denominator is uniformly separated from zero for sufficiently
small $|\xi|$, and dominated convergence applies before taking
$k\to\infty$. We obtain
\begin{align*}
\langle\mathcal{W}_{m}^{\mathrm{low}}h_k,1\rangle
&=-\sum_{j=0}^{m}\alpha_{0,j}^{(1)}\frac{\widehat{v\psi_0}(0)}{(2\pi)^4}
\lim_{\varepsilon\to 0^+}\int_{\mathbb{R}^4}
\frac{\overline{\widehat{v\psi_j}(\eta)}\,\widehat{h_k}(\eta)}{|\eta|^2+i\varepsilon}\,
f_{1}^{+}(|\eta|)\chi_0(|\eta|)\,d\eta \\
&=-\frac{\widehat{v\psi_0}(0)}{(2\pi)^4}
\lim_{\varepsilon\to 0^+}\int_{\mathbb{R}^4}
\frac{\overline{\widehat{v\psi_0}(\eta)}\,\widehat{h_k}(\eta)}{|\eta|^2+i\varepsilon}\,
f_{1}^{+}(|\eta|)\chi_0(|\eta|)\,d\eta
+\mathcal{O}\!\left(\int_{|\eta|\le\lambda_0}\frac{|f_{1}^{+}(|\eta|)|}{|\eta|}\,d\eta\right).
\end{align*}
Note that
$
f_{1}^{+}(|\eta|)=\frac{1}{|\eta|^2}\frac{1}{a\log|\eta|+z+\mathcal{O}(|\eta|^{0+})}
$
for a constant $a \neq 0$.
In four-dimensional polar coordinates the error is bounded by a constant times $\int_0^{\lambda_0}|\log r|^{-1}dr<\infty$.
Moreover, \eqref{eq:nonzero-vpsi0} implies that there exists a nonzero constant $c$ such that
\[
\frac{\widehat{v\psi_0}(0)}{(2\pi)^4}
\lim_{\varepsilon\to 0^+}\int_{\mathbb{R}^4}
\frac{\overline{\widehat{v\psi_0}(\eta)}\,\widehat{h_k}(\eta)}{|\eta|^2+i\varepsilon}\,
f_{1}^{+}(|\eta|)\chi_0(|\eta|)\,d\eta=c\log\log k+\mathcal{O}(1),\qquad k\gg1.
\]
It follows that
\[
\big|\langle\mathcal{W}_{m}^{\mathrm{low}}h_k,1\rangle\big|\longrightarrow\infty\qquad(k\to\infty).
\]
This contradicts the $L^1$-boundedness of $\mathcal{W}_{m}^{\mathrm{low}}$, thereby completing the proof.
\end{proof}

\section{The pure zero-eigenvalue case}\label{sec:eigenvalue}

\subsection{Inverse expansion and regular contribution}
We begin by separating the regular contribution from the finite-dimensional
eigenvalue block. The inverse expansion needed for this decomposition is
proved in Section~\ref{sec:exp of lemma}.

\begin{lemma}\label{lem:exp-of-M-inv}
Let $n\ge3$ and $|V(x)|\lesssim\langle x\rangle^{-\beta}$ with
$\beta>n+4$. Suppose that zero is an eigenvalue of $H$ and no
zero-energy resonance is present. Let $S_1$ be the orthogonal
projection onto $\ker T$. For all sufficiently small $\lambda>0$,
\[
(M^+(\lambda))^{-1}
=\lambda^{-2}S_1(D_2+G(\lambda))S_1+\mathcal N+\mathcal R(\lambda),
\]
where $\mathcal N$ is independent of $\lambda$ and absolutely
bounded, and $G(\lambda),\mathcal R(\lambda)=\widetilde{\mathcal O}(\lambda)$.
Moreover,
\[
S_1D_2S_1=w\mathcal P_e w,\qquad vS_1D_2S_1v=V\mathcal P_eV.
\]
If $\{\psi_j\}_{j=1}^d$ is an orthonormal basis of $\mathcal E$,
then $\{w\psi_j\}_{j=1}^d$ is a basis of $S_1L^2$.
\end{lemma}

We also localize on the perturbed spectral side. By intertwining,
\begin{equation}\label{eq-stationary formula-1}
\begin{aligned}
W_-^{\mathrm{low}}
&=\chi_0(\sqrt H/2)W_-^{\mathrm{low}}\\
&=\chi_0(\sqrt H/2)\chi_0(\sqrt{-\Delta})
-\chi_0(\sqrt H/2)\mathcal W^{\mathrm{low}}.
\end{aligned}
\end{equation}
Here $\chi_0(\sqrt H/2)$ denotes a smooth compactly
supported function of $H$ agreeing with the displayed expression
on its nonnegative spectrum and equal to $1$ near zero.
The cutoff kernel estimate in \cite[Lemma~2.4]{Yajima-1995-even}
gives $|C(x,u)|\lesssim_M\langle x-u\rangle^{-M}$ for every fixed
$M$. Thus this cutoff is admissible. Its role for $n\ge4$ is to
regularize the locally non-square-integrable free-resolvent kernel.
The first term on the right of \eqref{eq-stationary formula-1}
is admissible, so it remains to study the second term.

\begin{lemma}\label{lem:bound-of-re-part}
Suppose that $|V(x)|\lesssim\langle x\rangle^{-\beta}$ with
$\beta>n+4$. Let $\mathcal N$ and $\mathcal R(\lambda)$ be as in
Lemma~\ref{lem:exp-of-M-inv}. Then
\[
\mathcal W_R=\chi_0\big(\sqrt H/2\big)\circ
\frac1{\pi i}\int_0^\infty\lambda\chi_0(\lambda)
R_0^+(\lambda^2)v[\mathcal N+\mathcal R(\lambda)]v
(R_0^+(\lambda^2)-R_0^-(\lambda^2))\,d\lambda
\]
is bounded on $L^p(\mathbb R^n)$ for every $1\le p\le\infty$.
\end{lemma}

\begin{proof}[Proof in dimension three]
It suffices to prove that the integral operator, without the left
spectral cutoff, has an admissible kernel. Write
$r=|x-z|$ and $s=|w-y|$.
Using the explicit kernels
\[
R_0^\pm(\lambda^2)(x,y)=\frac{e^{\pm i\lambda|x-y|}}{4\pi|x-y|},
\]
the integral kernel, up to a fixed constant, is
\[
\iint \frac{v(z)v(w)}{rs}\int_0^\infty
\lambda\chi_0(\lambda)
\big(e^{i\lambda(r+s)}-e^{i\lambda(r-s)}\big)
[\mathcal N(z,w)+\mathcal R(\lambda)(z,w)]\,d\lambda\,dz\,dw.
\]
Set
\[
\mathcal A_0(z,w)=|\mathcal N(z,w)|,\qquad
\mathcal A_1(z,w)=\sup_{0\le j\le3}\sup_{0<\lambda<\lambda_0}
\lambda^{j-1}|\partial_\lambda^j\mathcal R(\lambda)(z,w)|.
\]
Both kernels give bounded operators on $L^2$. We shall repeatedly use
\begin{equation}\label{eq:repart-n3-weight}
\sup_{y\in\mathbb R^3}
\big\|v(w)(1+|w-y|^{-1})\big\|_{L^2_w}<\infty.
\end{equation}
Indeed, on $|w-y|<1$ use $V\in L^\infty$ and
$\int_{|u|<1}|u|^{-2}du<\infty$, and on its complement use
$V\in L^1$.

First consider $\mathcal R(\lambda)$. Since
$\lambda\mathcal R(\lambda)=\widetilde{\mathcal O}(\lambda^2)$,
we may apply Lemma~\ref{lem-oscillatory estimate} with $b=3/2$;
this uses only the three available derivatives and gives
\[
\left|\int_0^\infty e^{it\lambda}\lambda\chi_0(\lambda)
\mathcal R(\lambda)(z,w)\,d\lambda\right|
\lesssim\mathcal A_1(z,w)\langle t\rangle^{-5/2}.
\]
Thus its kernel is majorized by
\[
C\iint v(z)\mathcal A_1(z,w)v(w)
\frac{\langle r+s\rangle^{-5/2}+\langle r-s\rangle^{-5/2}}{rs}
\,dz\,dw.
\]
Polar integration gives
\[
\frac1s\int_0^\infty
r\big(\langle r+s\rangle^{-5/2}+\langle r-s\rangle^{-5/2}\big)\,dr
\lesssim1+s^{-1};
\]
the same estimate holds with $r$ and $s$ interchanged. Consequently
the integral of the absolute kernel in $x$ is bounded by
\[
C\iint v(z)\mathcal A_1(z,w)v(w)(1+|w-y|^{-1})\,dz\,dw
\le C\|v\|_2\|\mathcal A_1\|_{\mathbb B(L^2)}
\|v(\cdot)(1+|\cdot-y|^{-1})\|_2\lesssim1.
\]
The integral in $y$ is bounded in the same way, placing the
factor $1+|x-z|^{-1}$ on the left of $\mathcal A_1$. Hence the
$\mathcal R$ contribution is admissible.

Next consider the constant coefficient $\mathcal N$. Put
\[
F(t)=\int_0^\infty e^{it\lambda}\lambda\chi_0(\lambda)\,d\lambda,
\qquad k(r,s)=\frac{F(r+s)-F(r-s)}{rs}.
\]
We prove
\begin{equation}\label{eq:repart-n3-main-Schur}
\int_0^\infty |k(r,s)|r^2\,dr\lesssim1+s^{-1},\qquad
\int_0^\infty |k(r,s)|s^2\,ds\lesssim1+r^{-1}.
\end{equation}
On $r+s\le1$, boundedness of $F$ gives $|k(r,s)|\lesssim(rs)^{-1}$,
which immediately has the two bounds in
\eqref{eq:repart-n3-main-Schur} on this region.

On $r+s>1$, first treat the region $|r-s|\le(r+s)/2$.
Here $r\sim s$ and both are bounded below by $1/4$.
The oscillatory integral lemma with $b=1$ gives
\[
|k(r,s)|\lesssim\frac{
\langle r+s\rangle^{-2}+\langle r-s\rangle^{-2}}{rs}.
\]
Since $r/s$ and $s/r$ are bounded on this region, both radial
integrals in \eqref{eq:repart-n3-main-Schur} are uniformly bounded.
This also handles the possible vanishing of $r-s$ without dividing
by that phase.

It remains to consider $r+s>1$ and $|r-s|>(r+s)/2$.
For $t\ne0$, two integrations by parts give the exact identity
\[
F(t)=-\frac1{t^2}+H(t),\qquad
H(t)=-\frac1{t^2}\int_0^\infty e^{it\lambda}
\big(2\chi_0'(\lambda)+\lambda\chi_0''(\lambda)\big)\,d\lambda.
\]
The boundary term is $-t^{-2}$ because
$(\lambda\chi_0(\lambda))'|_{\lambda=0}=1$.
All derivatives in the remaining integral fall on $\chi_0$;
its amplitude is smooth and compactly supported away from zero.
Therefore, for every fixed $M$,
\[
|H(t)|\lesssim_M\langle t\rangle^{-M},\qquad |t|\ge1/2.
\]
The two boundary terms must be subtracted before taking absolute
values. They give
\[
\frac1{rs}\left(\frac1{(r-s)^2}-\frac1{(r+s)^2}\right)
=\frac4{(r-s)^2(r+s)^2}\lesssim(r+s)^{-4}.
\]
For the cutoff terms, since $|r-s|\sim r+s$ in this region,
\[
\frac{|H(r+s)|+|H(r-s)|}{rs}
\lesssim_M\frac{\langle r+s\rangle^{-M}}{rs}.
\]
Taking $M>2$, both contributions satisfy
\eqref{eq:repart-n3-main-Schur}: the boundary contribution has
uniformly bounded radial integrals on $r+s>1$, whereas the cutoff
contribution has radial integrals bounded by $C/s$ and $C/r$,
respectively. This proves \eqref{eq:repart-n3-main-Schur}.

The kernel of the $\mathcal N$ contribution is
$\iint v(z)\mathcal N(z,w)v(w)k(|x-z|,|w-y|)\,dz\,dw$,
up to a constant. Applying \eqref{eq:repart-n3-main-Schur} and
\eqref{eq:repart-n3-weight}, with $\mathcal A_0$ in place of
$\mathcal A_1$, proves that its $L^1$-norm in either variable is uniformly bounded
with respect to the other variable, just as for the remainder term.
Thus the entire integral kernel is admissible. Schur's test gives
boundedness on $L^p$ for every $1\le p\le\infty$; composing with
the admissible spectral cutoff preserves these bounds.
\end{proof}

\begin{proof}[Proof in dimensions $n\ge4$]
The $L^1$-bound follows by the same argument as in
Lemma~\ref{lem:L1-bound-of-K}, using the free-resolvent
estimates in dimension $n$. We therefore concentrate on
the $L^\infty$-bound.

Put
\[
r=|x-z|,\qquad s=|w-y|,\qquad
k_0=\left[\frac{n+3}{2}\right].
\]
Recall that the nonnegative kernels
\[
\mathcal A_0(z,w)=|\mathcal N(z,w)|,\qquad
\mathcal A_1(z,w)
=
\sup_{0\le k\le k_0}\sup_{0<\lambda<\lambda_0}
\lambda^{k-1}
\bigl|\partial_\lambda^k\mathcal R(\lambda)(z,w)\bigr|
\]
define bounded operators on $L^2$. Consequently,
\begin{equation}\label{eq:ge4-coefficient-bound}
\iint_{\mathbb R^{2n}}
|v(z)|\mathcal A_i(z,w)|v(w)|\,dz\,dw
\le \|\mathcal A_i\|_{2\to2}\|v\|_2^2,
\qquad i=0,1.
\end{equation}

We split the left free resolvent into its large- and
small-argument parts:
\[
R_0^+(\lambda^2)(x,z)
=
R_0^+(\lambda^2)(x,z)\widetilde\chi(\lambda r)
+
R_0^+(\lambda^2)(x,z)\chi(\lambda r).
\]

\medskip
\noindent\emph{The large-argument part.}
By the free-resolvent formulas in Subsection \ref{sec:free-r}, its kernel is a linear
combination of
\[
\begin{aligned}
\iint_{\mathbb R^{2n}}
\frac{v(z)v(w)}{r^{\frac{n-1}{2}}s^{\frac{n-1}{2}}}
\int_0^\infty
e^{i\lambda(r\pm s)}
\lambda^{n-2}\omega_1(\lambda r)\omega_\pm(\lambda s)
\bigl(\mathcal N+\mathcal R(\lambda)\bigr)(z,w)
\chi_0(\lambda)\,d\lambda\,dz\,dw,
\end{aligned}
\]
where $\omega_1$ includes the large-argument cutoff.
Fix $0<\varepsilon<1/2$. Since
$n-2\ge\frac{n-1}{2}+\varepsilon$, from \eqref{eq:est-for-omega}, the symbol estimates give
\[
\begin{aligned}
\left|\partial_\lambda^k\left[
\lambda^{n-2}\omega_1(\lambda r)\omega_\pm(\lambda s)
\bigl(\mathcal N+\mathcal R(\lambda)\bigr)(z,w)
\right]\right|
\lesssim
\bigl(\mathcal A_0(z,w)+\mathcal A_1(z,w)\bigr)
\lambda^{\frac{n-1}{2}+\varepsilon-k}
\end{aligned}
\]
for $ 0<\lambda<\lambda_0$ and $
0\le k\le k_0.$
Applying Lemma~\ref{lem-oscillatory estimate} with
$b=\frac{n-1}{2}+\varepsilon$, we obtain the kernel majorant
\[
C\iint_{\mathbb R^{2n}}
|v(z)v(w)|
\bigl(\mathcal A_0(z,w)+\mathcal A_1(z,w)\bigr)
\left(\frac{1}{r^{\frac{n-1}{2}}\langle r+s\rangle^{\frac{n+1}{2}+\varepsilon}s^{\frac{n-1}{2}}
}+\frac{1}{r^{\frac{n-1}{2}}\langle r-s\rangle^{\frac{n+1}{2}+\varepsilon}s^{\frac{n-1}{2}}
}\right)\,dz\,dw.
\]
Using polar coordinates and the elementary estimate
$
\int_0^\infty
\frac{t^{\frac{n-1}{2}}}
{\langle t\pm R\rangle^{\frac{n+1}{2}+\varepsilon}}
\,dt
\lesssim 1+R^{\frac{n-1}{2}}
$
for $R>0$,
we bound the integral of the absolute kernel in $y$ by
\[
C\iint_{\mathbb R^{2n}}
|v(z)v(w)|
\bigl(\mathcal A_0(z,w)+\mathcal A_1(z,w)\bigr)
\bigl(1+|x-z|^{-\frac{n-1}{2}}\bigr)\,dz\,dw.
\]
This is uniformly bounded in $x$ by Cauchy--Schwarz, since
\[
\sup_x
\bigl\|v(\cdot)
\bigl(1+|x-\cdot|^{-\frac{n-1}{2}}\bigr)\bigr\|_2
<\infty.
\]
Indeed, $|x-z|^{-(n-1)}$ is locally integrable and
$V\in L^1\cap L^\infty$. Interchanging $r$ and $s$
gives the corresponding bound for the integral in $x$.
Thus the large-argument contribution has an admissible
kernel.

\medskip
\noindent\emph{The small-argument expansion.}
By \eqref{eq:kernel-est-1},
\begin{equation}\label{eq:ge4-small-expansion}
\begin{aligned}
R_0^+(\lambda^2)(x,z)\chi(\lambda r)
&=
\sum_{j=0}^{\left[\frac{n-3}{2}\right]}
b_j\lambda^{2j}r^{2j+2-n}
+r_n(\lambda,r).
\end{aligned}
\end{equation}
Here the cutoff has been removed from the finite sum and its
contribution absorbed into $r_n$. More precisely, $r_n$ consists
of the original small-argument remainder and
$-\sum_j b_j\lambda^{2j}r^{2j+2-n}\widetilde\chi(\lambda r)$.
For the latter terms, $\lambda r\gtrsim1$ and
$n-2-2j\ge1$, so the same bound holds. Thus, for fixed
$0<\varepsilon<1/2$,
\begin{equation}\label{eq:ge4-small-remainder}
\bigl|\partial_\lambda^k r_n(\lambda,r)\bigr|
\lesssim_{k,\varepsilon}
\lambda^{n-2-\varepsilon-k}r^{-\varepsilon},
\qquad 0\le k\le k_0.
\end{equation}

Before estimating these terms, we record the effect of
the left spectral cutoff. Denote the kernel of
$\chi_0(\sqrt H/2)$ by $C(x,u)$. The estimate
stated before the lemma gives
\[
|C(x,u)|\lesssim_M\langle x-u\rangle^{-M}.
\]
Consequently, for every $0<d<n$,
\begin{equation}\label{eq:ge4-smoothing}
\sup_{x,z}\int_{\mathbb R^n}
|C(x,u)|\,|u-z|^{-d}\,du<\infty.
\end{equation}
This follows from local integrability near $u=z$ and
the uniform integrability of $C(x,\cdot)$.

We will also use
\begin{equation}\label{eq:ge4-right-integrability}
\int_{\mathbb R^n}
|w-y|^{-\frac{n-1}{2}}
\langle w-y\rangle^{-\frac{n+1}{2}-\varepsilon}\,dy
<\infty,
\end{equation}
uniformly in $w$. In polar coordinates, the integrand
with the radial measure behaves like
$s^{\frac{n-1}{2}}$ near zero and like
$s^{-1-\varepsilon}$ at infinity.

\medskip
\noindent\emph{The finite sum paired with $\mathcal N$.}
Let $B_j$ be the operator with kernel
\[
B_j(u,z)=b_j|u-z|^{-d_j},
\qquad d_j=n-2-2j\in(0,n).
\]
Moreover, writing
\[
C_j(x,z)=\int_{\mathbb R^n}C(x,u)B_j(u,z)\,du,
\]
we obtain from \eqref{eq:ge4-smoothing} that
\[
\sup_x\|C_j(x,\cdot)v\|_2\lesssim\|v\|_2.
\]
The finite sum now has coefficients independent of $\lambda$.
The free spectral-measure identity gives
\begin{equation}\label{eq:ge4-polynomial-multiplier}
\begin{aligned}
\frac1{\pi i}\int_0^\infty
\lambda^{2j+1}\chi_0(\lambda)
C_jv\mathcal Nv
(R_0^+-R_0^-)(\lambda^2)\,d\lambda
=C_jv\mathcal Nv(-\Delta)^j
\chi_0(\sqrt{-\Delta}).
\end{aligned}
\end{equation}
The Fourier multiplier on the right has a Schwartz
kernel. Cauchy--Schwarz and the $L^2$-boundedness of $\mathcal N$
therefore show that \eqref{eq:ge4-polynomial-multiplier}
is bounded on $L^\infty$. The left spectral cutoff is already
included in $C_j$.

\medskip
\noindent\emph{The finite sum paired with
$\mathcal R(\lambda)$.}
Using \eqref{eq:kernel-est-4}, we write the kernel as a sum of
\[
\frac1{\pi i}\iint C_j(x,z)v(z)v(w)s^{-\frac{n-1}{2}}
\int_0^\infty e^{\pm i\lambda s}
\lambda^{2j+\frac{n-1}{2}}\chi_0(\lambda)
\omega_\pm(\lambda s)\mathcal R(\lambda)(z,w)
\,d\lambda\,dz\,dw.
\]

The scalar
oscillatory amplitudes are
\[
\lambda^{2j+\frac{n-1}{2}}
\omega_\pm(\lambda s)
\mathcal R(\lambda)(z,w).
\]
Since
$\mathcal R(\lambda)=\widetilde{\mathcal O}(\lambda)$,
\[
\begin{aligned}
\left|\partial_\lambda^k\left[
\lambda^{2j+\frac{n-1}{2}}
\omega_\pm(\lambda s)
\mathcal R(\lambda)(z,w)
\right]\right|
    &\lesssim
\mathcal A_1(z,w)
\lambda^{2j+\frac{n+1}{2}-k}
\\
&\lesssim
\mathcal A_1(z,w)
\lambda^{\frac{n-1}{2}+\varepsilon-k},
\qquad 0\le k\le k_0.
\end{aligned}
\]
Lemma~\ref{lem-oscillatory estimate} therefore yields
the kernel majorant, up to a constant,
\[
\iint_{\mathbb R^{2n}}
|C_j(x,z)v(z)|
\mathcal A_1(z,w)|v(w)|
s^{-\frac{n-1}{2}}
\langle s\rangle^{-\frac{n+1}{2}-\varepsilon}
\,dz\,dw.
\]
Equations \eqref{eq:ge4-right-integrability} and
\eqref{eq:ge4-coefficient-bound} show that the $L^1$-norm of this kernel in $y$ is uniformly
bounded in $x$. This proves the $L^\infty$ bound.

\medskip
\noindent\emph{The remainder.}
We apply the same oscillatory estimate to $r_n$.
Including the factor $\lambda$ and the spectral
difference, the scalar amplitudes associated with
$r_n(\lambda,r)$ are
\[
\lambda^{\frac{n-1}{2}}
r_n(\lambda,r)\omega_\pm(\lambda s)\mathcal N(z,w)
\quad\text{and}\quad
\lambda^{\frac{n-1}{2}}
r_n(\lambda,r)\omega_\pm(\lambda s)
\mathcal R(\lambda)(z,w).
\]
By \eqref{eq:ge4-small-remainder}, their $k$-th
derivatives are bounded, respectively, by
\[
\begin{aligned}
&\mathcal A_0(z,w)r^{-\varepsilon}
\lambda^{\frac{n-1}{2}+n-2-\varepsilon-k},
\\
&\mathcal A_1(z,w)r^{-\varepsilon}
\lambda^{\frac{n-1}{2}+n-1-\varepsilon-k}.
\end{aligned}
\]
Since $n\ge4$ and $0<\varepsilon<1/2$, both satisfy
the bound
\[
C\bigl(\mathcal A_0(z,w)+\mathcal A_1(z,w)\bigr)
r^{-\varepsilon}
\lambda^{\frac{n-1}{2}+\varepsilon-k}.
\]
Another application of the oscillatory estimate gives, up to a constant,
\[
\iint_{\mathbb R^{2n}}
\left(\int_{\mathbb R^n}|C(x,u)|\,|u-z|^{-\varepsilon}\,du\right)|v(z)|
\bigl(\mathcal A_0(z,w)+\mathcal A_1(z,w)\bigr)|v(w)|
s^{-\frac{n-1}{2}}
\langle s\rangle^{-\frac{n+1}{2}-\varepsilon}
\,dz\,dw
\]
as a kernel majorant. The integral in parentheses is uniformly bounded in $x,z$
by \eqref{eq:ge4-smoothing}.
Integrating in $y$ and applying
\eqref{eq:ge4-right-integrability} and
\eqref{eq:ge4-coefficient-bound} proves the
$L^\infty$-bound for the remainder.

All applications of the oscillatory lemma use
$b=\frac{n-1}{2}+\varepsilon$, for which
$[b]+2=k_0$. Thus the available spectral derivatives
suffice. We have proved the $L^1$ and $L^\infty$ bounds
for $\mathcal W_R$, and interpolation completes the
proof for every $1\le p\le\infty$.
\end{proof}

\subsection{The eigenvalue contribution}
\label{sec:eigen-contribution}

In dimensions $n\ge3$, the contribution of the zero-energy eigenspace to the
low-energy part of the wave operator is
\[
\chi_0\big(\sqrt H/2\big)\circ\mathcal{W}_e+\chi_0\big(\sqrt H/2\big)\circ\mathcal{W}_{e,r},
\]
with
\begin{equation}\label{eq:def-Wesing-hd}
\begin{aligned}
\mathcal{W}_e&=\frac{1}{\pi i}\int_0^{+\infty} \lambda^{-1} \chi_0(\lambda)\,
R_0^+(\lambda^{2})V\mathcal{P}_eV
\bigl(R_0^+(\lambda^{2})-R_0^-(\lambda^{2})\bigr) \, d\lambda , \\
\mathcal{W}_{e,r}&=\frac{1}{\pi i}\int_0^{+\infty} \lambda^{-1} \chi_0(\lambda)\,
R_0^+(\lambda^{2})vS_1G(\lambda)S_1v
\bigl(R_0^+(\lambda^{2})-R_0^-(\lambda^{2})\bigr) \, d\lambda.
\end{aligned}
\end{equation}
Here $\mathcal W_e$ is the leading pole contribution, while
$\mathcal W_{e,r}$ retains the full correction within $S_1L^2$.
The moment spaces $\mathcal E_0\supset\mathcal E_1\supset\mathcal E_2$
were defined in the introduction. The proposition below concerns
$\mathcal W_e$ itself; the passage to the complete eigenvalue
contribution is discussed at the end of the section.

\begin{proposition}\label{lem:bound of we}
Let $n\ge3$ and $|V(x)|\lesssim\langle x\rangle^{-\beta}$ with
$\beta>n+4$.
Suppose that zero is an eigenvalue of $H=-\Delta+V$ and that no zero-energy
resonance is present. Then
\begin{enumerate}
\item[\emph{(i)}] $\mathcal{W}_e$ is bounded on
$L^p(\mathbb{R}^n)$ for $1\le p<\frac{n}{2}$ for $n\ge 5$ and for $1\le p\le 2$ for $n=3,4$; if
$\mathcal{E}\setminus\mathcal{E}_0\neq\emptyset$
(which is possible only when $n\ge5$), then it is unbounded on
$L^p(\mathbb{R}^n)$ for every
$\frac{n}{2}\le p\le\infty$;
\item[\emph{(ii)}] assume $\mathcal{E}=\mathcal{E}_0$; then
$\mathcal{W}_e$ is bounded on
$L^p(\mathbb{R}^n)$ for $1\le p<n$; if moreover
$\mathcal{E}_0\setminus\mathcal{E}_1\neq\emptyset$, it is unbounded on
$L^p(\mathbb{R}^n)$ for every $n\le p\le\infty$;
\item[\emph{(iii)}] assume $\mathcal{E}=\mathcal{E}_1$; then
$\mathcal{W}_e$ is bounded on
$L^p(\mathbb{R}^n)$ for $1\le p<\infty$; if moreover
$\mathcal{E}_1\setminus\mathcal{E}_2\neq\emptyset$, it is unbounded on
$L^\infty(\mathbb{R}^n)$;
\item[\emph{(iv)}] assume $\mathcal{E}=\mathcal{E}_2$; then
$\mathcal{W}_e$ is bounded on
$L^\infty(\mathbb{R}^n)$.
\end{enumerate}
\end{proposition}

\begin{proof}
The estimates for each eigenfunction contribution in
\cite[Sections 2--3]{Goldberg-Green-2016},
\cite[Proposition 2.1]{Goldberg-Green-2017}, and
\cite[Lemma 4.3]{Yajima-2016-3d} give the $L^p$-boundedness of
$\mathcal W_e$ for $1\le p\le2$ in dimensions $n\ge5$, $n=4$,
and $n=3$, respectively. These are estimates for the low-energy
singular integrals and apply under the stated decay assumption.
In dimensions three and four, the required zeroth-moment
cancellation holds automatically because $\mathcal E=\mathcal E_0$.

We prove the remaining assertions in four steps. First, we isolate
a finite-rank term. We then show that its remainder is bounded
for $2\le p\le\infty$, compute the asymptotics of the finite-rank
input profiles, and determine the exact ranges of $p$.
Throughout the proof, the operator identities are initially
understood on test functions, or through the indicated
eigenfunction pairings.

\medskip
\noindent\emph{Step 1: isolating the finite-rank term.}
Choose an orthonormal basis $\{\psi_j\}_{j=1}^k$ of
$\mathcal E=\mathcal P_eL^2(\mathbb R^n)$.
The zero-energy equation gives
\begin{equation}\label{eq:eigen-rep-hd}
\begin{gathered}
\psi_j=-(-\Delta)^{-1}(V\psi_j),\\
(-\Delta)^{-1}V\mathcal P_e=-\mathcal P_e,
\qquad \mathcal P_eV(-\Delta)^{-1}=-\mathcal P_e.
\end{gathered}
\end{equation}
Apply the free resolvent identity
\[
R_0^+(\lambda^2)=(-\Delta)^{-1}
+\lambda^2(-\Delta)^{-1}R_0^+(\lambda^2)
\]
to \eqref{eq:def-Wesing-hd}. This gives
$\mathcal W_e=\mathcal N_{4,m}+\mathcal N_{4,r}$, where
\begin{align*}
\mathcal N_{4,m}
&=\frac1{\pi i}\int_0^\infty\lambda^{-1}\chi_0(\lambda)
(-\Delta)^{-1}V\mathcal P_eV
(R_0^+-R_0^-)(\lambda^2)\,d\lambda,\\
\mathcal N_{4,r}
&=\frac1{\pi i}\int_0^\infty\lambda\chi_0(\lambda)
(-\Delta)^{-1}R_0^+(\lambda^2)V\mathcal P_eV
(R_0^+-R_0^-)(\lambda^2)\,d\lambda.
\end{align*}
The free spectral-measure identity is
\begin{equation}\label{eq:free-measure}
\frac{\lambda}{\pi i}(R_0^+-R_0^-)(\lambda^2)\varphi
=\bigl(\widehat\varphi\,d\sigma_\lambda\bigr)^\vee,
\end{equation}
where $d\sigma_\lambda$ is surface measure on $\{|\xi|=\lambda\}$.
Consequently,
\[
\frac1{\pi i}\int_0^\infty\lambda^{-1}\chi_0(\lambda)
(R_0^+-R_0^-)(\lambda^2)\,d\lambda
=(-\Delta)^{-1}\chi_0(\sqrt{-\Delta}).
\]
Using this identity under the pairing with $V\psi_j$, together
with \eqref{eq:eigen-rep-hd}, we obtain
\begin{align*}
\mathcal N_{4,m}f
=-\sum_{j=1}^k\psi_j
\bigl\langle(-\Delta)^{-1}\chi_0(\sqrt{-\Delta})f,V\psi_j\bigr\rangle
 =\sum_{j=1}^k\psi_j
\bigl\langle f,\chi_0(\sqrt{-\Delta})\psi_j\bigr\rangle.
\end{align*}
These pairings are justified by the decay of $V\psi_j$ and
$(-\Delta)^{-1}V\psi_j=-\psi_j\in L^2$.
Thus
\begin{equation}\label{eq:N4m-identity-hd}
\mathcal N_{4,m}=\mathcal P_e\chi_0(\sqrt{-\Delta}).
\end{equation}

\medskip
\noindent\emph{Step 2: boundedness of the remainder for
$2\le p\le\infty$.}
Split the factor $(-\Delta)^{-1}$ in $\mathcal N_{4,r}$ as
\[
(-\Delta)^{-1}
=\frac{\chi_0(\sqrt{-\Delta})}{-\Delta}
+\frac{\widetilde\chi_0(\sqrt{-\Delta})}{-\Delta},
\]
and denote the resulting terms by $\mathcal N_{4,r,0}$ and
$\mathcal N_{4,r,1}$, respectively.

For $\mathcal N_{4,r,1}$, the symbol
$\widetilde\chi_0(|\xi|)|\xi|^{-2}$ has an integrable inverse
Fourier transform: its kernel is $\mathcal{O}(|x|^{2-n})$ near zero and
decays to arbitrary order at infinity.
The stationary integral to its right is also bounded on every
$L^p$. Indeed, boundedness of the eigenfunctions gives
\[
V\mathcal P_eV=\sum_{j=1}^k(V\psi_j)\langle\cdot,V\psi_j\rangle,
\qquad |V\mathcal P_eV(z,w)|\lesssim|V(z)||V(w)|.
\]
The proof of Lemma~\ref{lem:bound-of-re-part} therefore applies
without its left spectral cutoff. Every local singularity
$|x-z|^{-d}$, $0<d<n$, is now integrated directly against $|V(z)|$,
using
\[
\sup_x\int_{\mathbb R^n}|V(z)|\,|x-z|^{-d}\,dz<\infty.
\]
The terms independent of $\lambda$ are treated by the same spectral-multiplier
identity, and the remaining terms by the same oscillatory
estimates. This proves the required bound for $\mathcal N_{4,r,1}$.

For $\mathcal N_{4,r,0}$, we identify its adjoint.
By \eqref{eq:free-measure} and the coarea formula,
\begin{equation}\label{eq:Foure express-hd}
\begin{aligned}
\mathcal F(\mathcal W_ef)(\xi)
={}&\sum_{j=1}^k\frac{\widehat{V\psi_j}(\xi)}{(2\pi)^n}
\lim_{\varepsilon\downarrow0}\int_{\mathbb R^n}
\frac{\overline{\widehat{V\psi_j}(\eta)}\,
\chi_0(|\eta|)\widehat f(\eta)}
{|\eta|^2(|\xi|^2-|\eta|^2-i\varepsilon)}\,d\eta.
\end{aligned}
\end{equation}
On the other hand,
\[
\mathcal N_{4,r,0}^*
=\frac1{\pi i}\int_0^\infty\lambda\chi_0(\lambda)
(R_0^+-R_0^-)(\lambda^2)V\mathcal P_eV
R_0^-(\lambda^2)\frac{\chi_0(\sqrt{-\Delta})}{-\Delta}\,d\lambda.
\]
Applying \eqref{eq:free-measure} to the outer spectral difference
sets $\lambda=|\xi|$ and yields
\begin{equation}\label{eq:dual-form-hd}
\begin{aligned}
\mathcal F(\mathcal N_{4,r,0}^*f)(\xi)
={}&\sum_{j=1}^k
\frac{\chi_0(|\xi|)\widehat{V\psi_j}(\xi)}{(2\pi)^n}
\lim_{\varepsilon\downarrow0}\int_{\mathbb R^n}
\frac{\overline{\widehat{V\psi_j}(\eta)}\,
\chi_0(|\eta|)\widehat f(\eta)}
{|\eta|^2(|\eta|^2-|\xi|^2+i\varepsilon)}\,d\eta.
\end{aligned}
\end{equation}
Comparing the denominators in these two formulas gives
\[
\mathcal N_{4,r,0}^*
=-\chi_0(\sqrt{-\Delta})\mathcal W_e.
\]
The right-hand side is bounded on $L^p$ for $1\le p\le2$ by
the estimates recalled at the start. Duality therefore gives
the boundedness of $\mathcal N_{4,r,0}$, and hence
$\mathcal N_{4,r}$, for $2\le p\le\infty$.
At $p=\infty$, the extension is the adjoint of the corresponding
$L^1$ operator and agrees with the stationary expression on test
functions.

\medskip
\noindent\emph{Step 3: moment asymptotics of the input profiles.}
We compare the large-$x$ expansions of $\psi_j$ and
$\chi_0(\sqrt{-\Delta})\psi_j$. Set
$b_0=((n-2)|\mathbb S^{n-1}|)^{-1}$.
By \eqref{eq:eigen-rep-hd},
\[
\psi_j(x)=-b_0\int_{\mathbb R^n}|x-y|^{2-n}(V\psi_j)(y)\,dy.
\]
The usual Newton-kernel estimate gives
$|\psi_j(x)|\lesssim\langle x\rangle^{2-n}$.
In particular, $\langle\cdot\rangle^{n+2}V\psi_j\in L^1$
under $\beta>n+4$. We use the moments
\begin{equation}\label{eq:moments-hd}
\begin{aligned}
&m_0^{(j)}=\int_{\mathbb R^n}V\psi_j\,dy,
\qquad d_i^{(j)}=\int_{\mathbb R^n}y_iV\psi_j\,dy,\\
&m_{i,l}^{(j)}=\int_{\mathbb R^n}y_iy_lV\psi_j\,dy,
\qquad m_{i,l,m}^{(j)}=\int_{\mathbb R^n}y_iy_ly_mV\psi_j\,dy,\\
&\bar m^{(j)}=\sum_i m_{i,i}^{(j)},
\qquad \mu_{i,l}^{(j)}=m_{i,l}^{(j)}-\frac{\delta_{il}}n\bar m^{(j)}.
\end{aligned}
\end{equation}
Thus $m_1^{(j)}=(d_1^{(j)},\ldots,d_n^{(j)})$, and
\begin{equation}\label{eq:moment-char-hd}
\begin{aligned}
\psi_j\in\mathcal E_0
&\iff m_0^{(j)}=0,\\
\psi_j\in\mathcal E_1
&\iff m_0^{(j)}=0\ \text{and}\ m_1^{(j)}=0,\\
\psi_j\in\mathcal E_2
&\iff \psi_j\in\mathcal E_1\ \text{and}\
\mu_{i,l}^{(j)}=0\ \text{for all }i,l.
\end{aligned}
\end{equation}
The last equivalence follows because the trace-free quadratic
polynomials $y_iy_l-|y|^2\delta_{il}/n$ span $\mathcal A_2$.

For $|y|\le|x|/2$, expand the Newton kernel to second order:
\[
\begin{aligned}
|x-y|^{2-n}
={}&|x|^{2-n}+(n-2)\frac{x\cdot y}{|x|^n}\\
&+\frac{n(n-2)}2\sum_{i,l}
\frac{x_ix_l}{|x|^{n+2}}
\left(y_iy_l-\frac{|y|^2}n\delta_{il}\right)
+\mathcal{O}(|x|^{-n-1}|y|^3).
\end{aligned}
\]
The integral over $|y|>|x|/2$, as well as the tails of the
Taylor terms, is $\mathcal{O}(|x|^{2-\beta})$ by boundedness of $\psi_j$
and the decay of $V$. Since $\beta>n+4$, we obtain
\begin{equation}\label{eq:Psij-asymp-hd}
\begin{aligned}
\psi_j(x)=-b_0\bigg[m_0^{(j)}|x|^{2-n}
+(n-2)\frac{x\cdot m_1^{(j)}}{|x|^n}+\frac{n(n-2)}2\sum_{i,l}\mu_{i,l}^{(j)}
\frac{x_ix_l}{|x|^{n+2}}\bigg]
+\mathcal{O}(|x|^{-n-1}).
\end{aligned}
\end{equation}
For $n=3,4$, the first term is incompatible with $\psi_j\in L^2$
unless $m_0^{(j)}=0$. This recovers $\mathcal E=\mathcal E_0$.

To identify the effect of the spectral cutoff, use
$\widehat\psi_j(\eta)=-|\eta|^{-2}\widehat{V\psi_j}(\eta)$ and
expand
\begin{equation}\label{eq:VPsi-taylor-hd}
\begin{aligned}
\widehat{V\psi_j}(\eta)
=m_0^{(j)}-i\eta\cdot m_1^{(j)}
-\frac{\bar m^{(j)}}{2n}|\eta|^2
-\frac12\sum_{i,l}\mu_{i,l}^{(j)}\eta_i\eta_l+\frac i6\sum_{i,l,m}m_{i,l,m}^{(j)}\eta_i\eta_l\eta_m
+E_j(\eta).
\end{aligned}
\end{equation}
Taylor's theorem gives
$|\partial_\eta^\alpha E_j(\eta)|\lesssim|\eta|^{4-|\alpha|}$
for $|\alpha|\le4$, while its derivatives of orders
$4<|\alpha|\le n+1$ are bounded near zero.

The constant and linear terms reproduce the first two terms of
\eqref{eq:Psij-asymp-hd}. For the quadratic term, use the
distributional identity
\begin{equation}\label{eq:Riesz-identity-hd}
\begin{aligned}
\mathcal F^{-1}(\eta_i\eta_l|\eta|^{-2})
&=-\partial_i\partial_l(b_0|x|^{2-n})\\
&=-b_0(n-2)\,\mathrm{p.v.}\,
\frac{nx_ix_l-\delta_{il}|x|^2}{|x|^{n+2}}
+\frac{\delta_{il}}n\delta_0.
\end{aligned}
\end{equation}
The delta coefficient follows by taking the trace and using
$-\Delta(b_0|x|^{2-n})=\delta_0$.
Contraction with the trace-free matrix $(\mu_{i,l}^{(j)})$
eliminates both terms containing $\delta_{il}$ and leaves the
harmonic quadratic term in \eqref{eq:Psij-asymp-hd}.
In contrast, the trace term in \eqref{eq:VPsi-taylor-hd} cancels
the denominator $|\eta|^2$ and contributes the Schwartz function
$\frac{\bar m^{(j)}}{2n}\mathcal F^{-1}(\chi_0(|\cdot|))$.

For these homogeneous symbols, insertion of $\chi_0$ changes
their kernels by a rapidly decaying term away from the origin.
The cutoff cubic terms are $\mathcal{O}(|x|^{-n-1})$ at infinity.
Finally, the compactly supported symbol
$E_j(\eta)\chi_0(|\eta|)|\eta|^{-2}$ has derivatives bounded by
$C|\eta|^{2-|\alpha|}$ for $|\alpha|\le n+1$.
Its derivatives of order $n+1$ are integrable, so integration by
parts gives the same $\mathcal{O}(|x|^{-n-1})$ bound for its inverse
Fourier transform. We conclude that
\begin{equation}\label{eq:chiPsij-asymp-hd}
\begin{aligned}
\chi_0(\sqrt{-\Delta})\psi_j(x)
={}&-b_0m_0^{(j)}|x|^{2-n}
-b_0(n-2)\frac{x\cdot m_1^{(j)}}{|x|^n}\\
&-\frac{b_0n(n-2)}2\sum_{i,l}\mu_{i,l}^{(j)}
\frac{x_ix_l}{|x|^{n+2}}\\
&+\frac{\bar m^{(j)}}{2n}\mathcal F^{-1}(\chi_0(|\cdot|))(x)
+\mathcal{O}(|x|^{-n-1}),\qquad |x|\to\infty.
\end{aligned}
\end{equation}
Thus the first three possible nonzero homogeneous terms have
orders $|x|^{2-n}$, $|x|^{1-n}$, and $|x|^{-n}$, respectively.
When $\psi_j\in\mathcal E_2$, all three vanish and the input
profile belongs to $L^1$.

\medskip
\noindent\emph{Step 4: the $L^p$ classification.}
For $2\le p\le\infty$, every output profile $\psi_j$ belongs
to $L^p$, because $\psi_j\in L^2\cap L^\infty$.
Therefore \eqref{eq:N4m-identity-hd} implies
\begin{equation}\label{eq:we-finite-rank-criterion}
\mathcal N_{4,m}:L^p\longrightarrow L^p
\quad\Longleftrightarrow\quad
\chi_0(\sqrt{-\Delta})\psi_j\in L^{p'}
\ \text{for every }j.
\end{equation}
Sufficiency follows from H\"older's inequality.
For necessity, choose $\phi_l\in C_c^\infty$ such that
$\langle\psi_j,\phi_k\rangle=\delta_{jk}$, which is possible
by linear independence. Pairing $\mathcal N_{4,m}f$ with $\phi_k$
isolates the functional
$\langle f,\chi_0(\sqrt{-\Delta})\psi_k\rangle$.
Its boundedness on compactly supported test functions gives
the required $L^{p'}$ membership, including $L^1$ when $p=\infty$.
By Step~2, the same criterion holds for $\mathcal W_e$.

For use in the necessity statements, a nonzero homogeneous term
$|x|^{-\gamma}A(x/|x|)$, with continuous nonzero $A$, cannot
belong to $L^q$ when $\gamma q\le n$.
Indeed, $|A|$ is bounded below on an open subset of the sphere,
and the radial integral there diverges. The faster-decaying errors
in \eqref{eq:chiPsij-asymp-hd} do not change this conclusion.

\begin{enumerate}[label=(\roman*)]
\item In general, the input profiles are
$\mathcal{O}(\langle x\rangle^{2-n})$. For $n\ge5$, they belong to
$L^{p'}$ when $p<n/2$, giving the bounds for $2\le p<n/2$.
Together with the initial $1\le p\le2$ estimates, this proves
the asserted boundedness range. If
$\mathcal E\setminus\mathcal E_0\ne\varnothing$, then some
$m_0^{(j)}\ne0$, so that input profile fails to lie in $L^{p'}$
for $p'\le n/(n-2)$, equivalently $p\ge n/2$.
The criterion \eqref{eq:we-finite-rank-criterion} proves precisely
the claimed unboundedness range.

\item If $\mathcal E=\mathcal E_0$, the input profiles are
$\mathcal{O}(\langle x\rangle^{1-n})$ and belong to $L^{p'}$ for
$p'>n/(n-1)$, equivalently $p<n$.
This gives boundedness for $2\le p<n$ and, with the initial
estimates, for $1\le p<n$.
If $\mathcal E_0\setminus\mathcal E_1\ne\varnothing$, some
$m_1^{(j)}\ne0$. The corresponding nonzero dipole term excludes
$L^{p'}$ membership for $p'\le n/(n-1)$, and hence excludes
$L^p$-boundedness for every $p\ge n$.

\item If $\mathcal E=\mathcal E_1$, the input profiles are
$\mathcal{O}(\langle x\rangle^{-n})$, so they belong to every $L^{p'}$
with $p'>1$. This proves boundedness for $2\le p<\infty$;
the initial estimates include the remaining range.
If $\mathcal E_1\setminus\mathcal E_2\ne\varnothing$, some
trace-free matrix $(\mu_{i,l}^{(j)})$ is nonzero. Its quadratic
form is a nonzero spherical harmonic, so the corresponding
$|x|^{-n}$ term is not integrable. Thus that input profile
does not belong to $L^1$, and $\mathcal W_e$ is unbounded on
$L^\infty$.

\item If $\mathcal E=\mathcal E_2$, every input profile belongs
to $L^1$ by Step~3. In particular,
\[
\|\mathcal N_{4,m}f\|_\infty
\le\sum_{j=1}^k\|\psi_j\|_\infty
\|\chi_0(\sqrt{-\Delta})\psi_j\|_1\,\|f\|_\infty.
\]
Together with Step~2, this proves the $L^\infty$-bound.
\end{enumerate}
This completes the proof.
\end{proof}

\subsection{Completion of the proof of Theorem~\ref{thm:main-infty}}
\label{sec:final-comparison}

Proposition~\ref{lem:bound of we} treats the leading term
$\mathcal W_e$. We now include $\mathcal W_{e,r}$. The sum again
splits into a finite-rank operator and a bounded remainder.
We show that the correction to each input function is integrable
at the relevant critical exponent, so it cannot cancel a
nonintegrable leading term.
Throughout, fix an orthonormal basis $\{\psi_i\}_{i=1}^d$ of
$\mathcal E$ and put
\[
q_i=\chi_0(\sqrt{-\Delta})\psi_i,
\qquad k_0=\left[\frac{n+3}{2}\right].
\]

We first deduce a Fourier multiplier estimate from
Lemma~\ref{lem-oscillatory estimate}.
\begin{lemma}\label{lem:vanishing-radial-multiplier}
Let $n\ge3$, $k_0=[(n+3)/2]$, and let $b$ be compactly supported
in $[0,\infty)$ and smooth away from zero. Suppose that
\[
|\partial_\lambda^k b(\lambda)|\lesssim_k\lambda^{\sigma-k},
\qquad 0\le k\le k_0,
\]
near zero, for some $\sigma>0$. Then $b(\sqrt{-\Delta})$ has
an integrable convolution kernel and is bounded on $L^p$ for
every $1\le p\le\infty$.
\end{lemma}
\begin{proof}
The kernel is bounded for $r=|x-y|\le1$, since
$b(|\xi|)\in L^1(\mathbb R^n)$. For $r\ge1$, the free
spectral-measure identity \eqref{eq:free-measure} and
\eqref{eq:kernel-est-4} express it as a linear combination of
\[
r^{-\frac{n-1}{2}}\int_0^\infty e^{\pm i\lambda r}
\lambda^{\frac{n-1}{2}}b(\lambda)\omega_\pm(\lambda r)\,d\lambda.
\]
Choose $0<\varepsilon<\min\{\sigma,1/2\}$.
By \eqref{eq:est-for-omega} and the product rule,
\[
\left|\partial_\lambda^k
\bigl[\lambda^{\frac{n-1}{2}}b(\lambda)
\omega_\pm(\lambda r)\bigr]\right|
\lesssim\lambda^{\frac{n-1}{2}+\varepsilon-k},
\qquad 0\le k\le k_0,
\]
uniformly in $r$. Apply Lemma~\ref{lem-oscillatory estimate}
with power exponent $\frac{n-1}{2}+\varepsilon$ and logarithmic
exponent zero, inserting a smooth cutoff equal to one on the
support of $b$.
It follows that the kernel is bounded by
$Cr^{-\frac{n-1}{2}}r^{-\frac{n+1}{2}-\varepsilon}
=Cr^{-n-\varepsilon}$ for $r\ge1$, proving its integrability.
\end{proof}

\medskip
\noindent\emph{The mixed eigenfunction terms.}
Set $\mathcal P_{i,j}f=\psi_i\langle f,\psi_j\rangle$.
Write $\mathcal W_{i,j}$ for the operator obtained from
$\mathcal W_e$ by replacing $\mathcal P_e$ with $\mathcal P_{i,j}$.
The kernel estimates recalled at the beginning of Subsection~\ref{sec:eigen-contribution}
give
\begin{equation}\label{eq:mixed-lower-bounds}
\mathcal W_{i,j}:L^p\longrightarrow L^p,
\qquad 1\le p\le2.
\end{equation}
For clarity, the passage from diagonal to mixed terms can also be
made by polarization. With $\mathcal P_u f=u\langle f,u\rangle$,
\[
\mathcal P_{i,j}
=\frac14\sum_{m=0}^3 \mathrm{i}^{m}\mathcal P_{\psi_i+\mathrm{i}^{m}\psi_j}.
\]
Each profile in this identity is a zero-energy eigenfunction, so
the individual-profile estimates apply to every summand.

The free resolvent identity used in Step~1 of Proposition~\ref{lem:bound of we} gives
\begin{equation}\label{eq:mixed-main-remainder}
\mathcal W_{i,j}f
=\psi_i\langle f,q_j\rangle+\mathcal B_{i,j}f,
\end{equation}
where
\[
\mathcal B_{i,j}
=\frac1{\pi i}\int_0^\infty\lambda\chi_0(\lambda)
(-\Delta)^{-1}R_0^+(\lambda^2)V\mathcal P_{i,j}V
(R_0^+-R_0^-)(\lambda^2)\,d\lambda.
\]
The high-frequency part of $(-\Delta)^{-1}$ is treated exactly as
in Step~2 of Proposition~\ref{lem:bound of we}, since
$|V\mathcal P_{i,j}V(z,w)|\lesssim|V(z)||V(w)|$.
For its low-frequency part, the computation in
\eqref{eq:Foure express-hd}--\eqref{eq:dual-form-hd}, with the
indices interchanged on taking the adjoint, gives
\[
\mathcal B_{i,j,0}^{*}
=-\chi_0(\sqrt{-\Delta})\mathcal W_{j,i}.
\]
Thus \eqref{eq:mixed-lower-bounds} and duality imply
\begin{equation}\label{eq:mixed-remainder-upper}
\mathcal B_{i,j}:L^p\longrightarrow L^p,
\qquad 2\le p\le\infty.
\end{equation}
The endpoint statement refers to the integral-kernel extension;
it follows from the $L^1$-bounded adjoint as in that proof.

\medskip
\noindent\emph{The finite-rank reduction.}
Extend each coefficient $g_{ij}$ smoothly away from zero,
agreeing with its original values on the support of $\chi_0$.
The extra cutoff $\chi_0(\lambda/2)$ will make every multiplier
compactly supported. Since it equals one on the support of
$\chi_0(\lambda)$, it does not change the stationary integrals.

The expansion of $G$ in the
basis $f_i=w\psi_i$ is
\[
S_1G(\lambda)S_1
=\sum_{i,j=1}^d g_{ij}(\lambda)f_i\langle\cdot,f_j\rangle,
\qquad g_{ij}=\mathcal O_{k_0}(\lambda).
\]
Since the spectral difference is supported on $|\xi|=\lambda$,
\begin{equation}\label{eq:correction-factorization}
\mathcal W_{e,r}
=\sum_{i,j=1}^d\mathcal W_{i,j}g_{ij}(\sqrt{-\Delta})\chi_0(\sqrt{-\Delta}/2).
\end{equation}
All these scalar Fourier multipliers $g_{ij}(\sqrt{-\Delta})\chi_0(\sqrt{-\Delta}/2)$ are bounded on every $L^p$ by
Lemma~\ref{lem:vanishing-radial-multiplier}.
Combining \eqref{eq:mixed-main-remainder} and
\eqref{eq:correction-factorization} gives the exact decomposition
\begin{equation}\label{eq:full-eigen-finite-rank}
\mathcal W_e+\mathcal W_{e,r}
=\mathcal M+\mathcal B,
\qquad
\mathcal Mf=\sum_{i=1}^d\psi_i\langle f,h_i\rangle,
\end{equation}
where
\begin{equation}\label{eq:corrected-input-profiles}
\begin{aligned}
h_i&=q_i+\sum_{j=1}^d\overline{g_{ij}}(\sqrt{-\Delta})q_j,\\
\mathcal B&=\sum_i\mathcal B_{i,i}
+\sum_{i,j}\mathcal B_{i,j}
g_{ij}(\sqrt{-\Delta})\chi_0(\sqrt{-\Delta}/2).
\end{aligned}
\end{equation}
In the formula for $h_i$, the extra cutoff is redundant because
$\chi_0(\sqrt{-\Delta}/2)q_j=q_j$.
The conjugation in this formula follows from our convention that
the inner product is linear in its first argument.
By \eqref{eq:mixed-remainder-upper}, $\mathcal B$ is bounded on
$L^p$ for $2\le p\le\infty$.
Since $\mathcal P_e$ is the spectral projection of $H$ at zero
and $\chi_0(0)=1$, the spectral theorem gives
\begin{equation}\label{eq:cutoff-preserves-eigenspace}
\chi_0(\sqrt H/2)\mathcal P_e
=\chi_0(0)\mathcal P_e=\mathcal P_e.
\end{equation}
In particular, $\chi_0(\sqrt H/2)\psi_i=\psi_i$ for every $i$.
Since the range of $\mathcal M$ is contained in $\mathcal E$, we obtain
\[
\chi_0(\sqrt H/2)(\mathcal M+\mathcal B)
=\mathcal M+\chi_0(\sqrt H/2)\mathcal B.
\]
Consequently, \eqref{eq-stationary formula-1} yields
\begin{equation}\label{eq:wave-finite-rank-reduction}
W_-^{\mathrm{low}}=\mathcal B_0-\mathcal M,
\qquad
\mathcal B_0:L^p\longrightarrow L^p
\quad(2\le p\le\infty).
\end{equation}
Here
\[
\mathcal B_0
=\chi_0(\sqrt H/2)\chi_0(\sqrt{-\Delta})
-\mathcal W_R-\chi_0(\sqrt H/2)\mathcal B.
\]
Its boundedness follows from Lemma~\ref{lem:bound-of-re-part},
the boundedness of $\mathcal B$, and the admissibility of the spectral
cutoffs established after \eqref{eq-stationary formula-1}.
Thus the left spectral cutoff remains in the bounded remainder
and acts as the identity on the finite-rank term.

\medskip
\noindent\emph{The comparison in the critical dual spaces.}
The additional vanishing of $g_{ij}$ at zero improves the
integrability of $h_i-q_i$. The following lemma makes this
improvement precise, using the spaces $\mathcal E_\ell$ of the
introduction and the convention $\mathcal E_{-1}=\mathcal E$.

\begin{lemma}\label{lem:critical-profile-correction}
Fix $0<\delta<1/2$ and $\ell\in\{-1,0,1\}$, and assume that
$\mathcal E=\mathcal E_\ell$. Then, for every $j$,
\begin{equation}\label{eq:fractional-profile-decay}
\left|(-\Delta)^{\delta/2}q_j(x)\right|
\lesssim\langle x\rangle^{-(n+\ell-1+\delta)}.
\end{equation}
Consequently, with
\[
q_*:=\frac{n}{n+\ell-1},
\]
one has
\begin{equation}\label{eq:critical-correction-integrable}
h_i-q_i\in L^{q_*}(\mathbb R^n),\qquad 1\le i\le d.
\end{equation}
\end{lemma}
\begin{proof}
Since $\psi_j\in L^\infty$, the decay assumption gives
$|V(x)\psi_j(x)|\lesssim\langle x\rangle^{-\beta}$,
where $\beta>n+4$.
By the eigenfunction equation and the definition of $q_j$,
\[
\mathcal F\bigl((-\Delta)^{\delta/2}q_j\bigr)(\xi)
=-|\xi|^{\delta-2}\chi_0(|\xi|)\widehat{V\psi_j}(\xi).
\]
Use the Fourier moment expansion \eqref{eq:VPsi-taylor-hd}
from Proposition~\ref{lem:bound of we}.
Under $\mathcal E=\mathcal E_\ell$, its terms of degrees at most
$\ell$ vanish, with no cancellation required when $\ell=-1$.
Multiplication by $|\xi|^{\delta-2}$ therefore makes the first
possible homogeneous term have degree $\ell-1+\delta$.
The same inverse-Fourier argument used to obtain
\eqref{eq:chiPsij-asymp-hd} gives
\[
\bigl|(-\Delta)^{\delta/2}q_j(x)\bigr|
\lesssim\langle x\rangle^{-(n+\ell-1+\delta)}.
\]
Here the higher Taylor terms decay faster, and the Taylor
remainder is $\mathcal{O}(|x|^{-n-1})$, which is sufficient
since $\ell\le1$ and $\delta<1/2$.
Thus $(-\Delta)^{\delta/2}q_j\in L^{q_*}$ because
$(n+\ell-1+\delta)q_*>n$.

Finally, the compactly supported symbol
$\lambda^{-\delta}\overline{g_{ij}(\lambda)}\chi_0(\lambda/2)$
belongs to $\mathcal O_{k_0}(\lambda^{1-\delta})$ near zero.
Lemma~\ref{lem:vanishing-radial-multiplier} and the identity
\[
\overline{g_{ij}}(\sqrt{-\Delta})q_j
=\left[\lambda^{-\delta}\overline{g_{ij}(\lambda)}
\chi_0(\lambda/2)\right](\sqrt{-\Delta})
(-\Delta)^{\delta/2}q_j
\]
show that every summand in $h_i-q_i$ belongs to $L^{q_*}$.
This proves \eqref{eq:critical-correction-integrable}.
\end{proof}

\begin{proof}[Proof of Theorem~\ref{thm:main-infty}]
We work with $W_-^{\mathrm{low}}$; complex conjugation gives the
same conclusions for $W_+^{\mathrm{low}}$.
For $1\le p\le2$, boundedness follows from
\eqref{eq:mixed-lower-bounds}, \eqref{eq:correction-factorization},
Lemma~\ref{lem:bound-of-re-part}, and the $L^p$-boundedness of
the left spectral cutoff $\chi_0(\sqrt H/2)$ in
\eqref{eq-stationary formula-1}.
For $2\le p\le\infty$, every $\psi_i$ belongs to $L^p$ because
$\psi_i\in L^2\cap L^\infty$.
We use the finite-rank reduction
\eqref{eq:wave-finite-rank-reduction} to prove boundedness and
then identify the operator responsible for unboundedness.

\medskip
\noindent\emph{Boundedness.}
Step~3 of Proposition~\ref{lem:bound of we} gives
$q_i\in L^{p'}$ in each claimed range, including $L^1$ under
$\mathcal E=\mathcal E_2$.
The compactly supported multipliers
$\overline{g_{ij}}(\sqrt{-\Delta})\chi_0(\sqrt{-\Delta}/2)$ preserve these
spaces, so \eqref{eq:corrected-input-profiles} gives the same
membership for $h_i$. Hence $\mathcal M$ is bounded in these ranges.
By \eqref{eq:cutoff-preserves-eigenspace}, its contribution after
applying the left spectral cutoff is exactly $\mathcal M$.
Together with the boundedness of $\mathcal B_0$, this proves all
the stated boundedness ranges.

\medskip
\noindent\emph{Necessity.}
Fix $2\le p\le\infty$. By
\eqref{eq:wave-finite-rank-reduction}, boundedness of
$W_-^{\mathrm{low}}$ is equivalent to boundedness of $\mathcal M$,
since $\mathcal B_0$ is bounded and the left spectral cutoff acts
as the identity on the range of $\mathcal M$. It therefore suffices
to show that one of its input profiles $h_i$ does not belong to
$L^{p'}$.

To justify this last implication, choose compactly supported smooth
functions $\phi_i$ with
$\langle\psi_j,\phi_i\rangle=\delta_{ji}$, as in Step~4 of
Proposition~\ref{lem:bound of we}. For compactly supported smooth $f$,
\[
\bigl\langle\chi_0(\sqrt H/2)\mathcal Mf,\phi_i\bigr\rangle
=\langle\mathcal Mf,\phi_i\rangle
=\langle f,h_i\rangle.
\]
If $\mathcal M$ were bounded on $L^p$, H\"older's inequality would give
\[
|\langle f,h_i\rangle|
\le\|\mathcal Mf\|_{L^p}\|\phi_i\|_{L^{p'}}
\lesssim\|f\|_{L^p}.
\]
This forces $h_i\in L^{p'}$; at $p=\infty$, the same estimate
on compactly supported test functions forces $h_i\in L^1$.
We now verify the failure of this necessary condition in each case.

First suppose that
$\mathcal E\setminus\mathcal E_0\ne\varnothing$.
Then $n\ge5$ and some $q_i$ has a nonzero term of degree $2-n$
at infinity, so $q_i\notin L^{n/(n-2)}$.
Lemma~\ref{lem:critical-profile-correction} with $\ell=-1$ gives
$h_i-q_i\in L^{n/(n-2)}$. Hence $h_i\notin L^{n/(n-2)}$,
and \eqref{eq:wave-finite-rank-reduction} excludes boundedness on
$L^{n/2}$.

If $\mathcal E=\mathcal E_0$ but
$\mathcal E_0\setminus\mathcal E_1\ne\varnothing$,
some $q_i$ has a nonzero dipole term and does not belong to
$L^{n/(n-1)}$.
The same lemma with $\ell=0$ shows that its correction belongs
to that space. Thus $h_i\notin L^{n/(n-1)}$, and boundedness
on $L^n$ is impossible.
In both cases the critical exponent is strictly larger than two.
Boundedness at any larger exponent, including infinity, would
imply boundedness at the critical exponent by interpolation with
the $L^2$-bound, a contradiction.

Finally, suppose that $\mathcal E=\mathcal E_1$ but
$\mathcal E_1\setminus\mathcal E_2\ne\varnothing$.
Some $q_i$ has a nonzero harmonic quadratic term of order
$|x|^{-n}$ and therefore fails to lie in $L^1$.
With $\ell=1$, Lemma~\ref{lem:critical-profile-correction} gives
$h_i-q_i\in L^1$, so $h_i\notin L^1$.
The finite-rank reduction excludes $L^\infty$-boundedness.
All assertions of the theorem follow.
\end{proof}

\section{Proof of Lemma~\ref{lem:exp-of-M-inv}}
\label{sec:exp of lemma}
\label{sec:block-inverse}

We prove the inverse expansion used in Section~\ref{sec:eigenvalue},
retaining the notation of the first three sections. In particular,
$G_1$ is the free-resolvent coefficient appearing in the term
$\lambda^2vG_1v$ of \eqref{eq:exp of M}. All inverses on projected
subspaces are extended by zero to $L^2$, and all
$\widetilde{\mathcal O}$ estimates have the absolute-kernel meaning
specified there, with derivatives through order $[(n+3)/2]$.

\begin{proof}
Put $Q=I-S_1$. We first justify the inverses and absolute bounds
used in the block-matrix argument.

\medskip
\noindent\emph{The inverses on $QL^2$ and $S_1L^2$.}
Since $vG_0v$ is compact and $U^2=I$, the operator
$T=U+vG_0v$ is self-adjoint and Fredholm. Thus $S_1$ has finite
rank, and $QTQ$ is invertible on $QL^2$. We write
\[
D_0=Q\bigl(QTQ|_{QL^2}\bigr)^{-1}Q.
\]
For $f\in S_1L^2$, set $\psi_f=-G_0vf$. The identity $Tf=0$ gives
\[
f=w\psi_f,\qquad V\psi_f=vf,\qquad H\psi_f=0,
\]
where $w=Uv$. By the pure-eigenvalue hypothesis, $\psi_f\in L^2$;
conversely, every $L^2$ zero mode arises in this way.

In dimensions three and four, $vf\in L^1$ and hence
$\widehat{vf}$ is continuous at zero. Since
\[
\widehat{\psi_f}(\xi)=-|\xi|^{-2}\widehat{vf}(\xi),
\]
the condition $\psi_f\in L^2$ forces $\widehat{vf}(0)=0$:
otherwise its squared $L^2$ norm would contain a positive multiple
of the divergent integral $\int_0^1r^{n-5}\,dr$. Consequently,
\begin{equation}\label{eq:lemma-inv-cancellation}
S_1P=PS_1=0,\qquad n=3,4.
\end{equation}
In particular, $S_2=S_1$ in the four-dimensional notation of Section~3.

For $f,g\in S_1L^2$, we have
\begin{equation}\label{eq:lemma-inv-positive}
\langle vG_1vf,g\rangle
=\frac1{(2\pi)^n}\int_{\mathbb R^n}
\frac{\widehat{vf}(\xi)\overline{\widehat{vg}(\xi)}}{|\xi|^4}\,d\xi
=\langle\psi_f,\psi_g\rangle.
\end{equation}
For $n\ge5$, this follows from the Fourier representation of $G_1$.
For $n=3,4$, compress the sandwiched operator
\[
v\frac{G_0-R_0(-\kappa^2)}{\kappa^2}v
\]
to $S_1L^2$ and let $\kappa\downarrow0$. The free multiplier is
$[|\xi|^2(|\xi|^2+\kappa^2)]^{-1}$, and the constant-kernel terms
in the small-$\kappa$ expansion vanish by
\eqref{eq:lemma-inv-cancellation}. Moreover, the decay assumption
implies $\int |x|\,|vf(x)|\,dx<\infty$, so
$\widehat{vf}(\xi)=\mathcal{O}(|\xi|)$ near zero, and similarly for $g$.
These bounds give an integrable majorant near zero; at infinity,
Plancherel and Cauchy--Schwarz suffice. This proves
\eqref{eq:lemma-inv-positive} in all dimensions under consideration.
Since $f\mapsto\psi_f$ is injective, $S_1vG_1vS_1$ is positive
definite on the finite-dimensional space $S_1L^2$. Hence
\[
D_2=\bigl(S_1vG_1vS_1|_{S_1L^2}\bigr)^{-1},
\qquad D_2=S_1D_2S_1.
\]

We next identify the eigenvalue coefficient. Choose an orthonormal
basis $\{\psi_j\}_{j=1}^d$ of $\mathcal E=\ker H$ and put
$f_j=w\psi_j$. Then $\{f_j\}_{j=1}^d$ is a basis of $S_1L^2$,
although it need not be orthonormal. Accordingly,
\[
S_1=\sum_{i,j=1}^d
\alpha_{i,j}f_i\langle\cdot,f_j\rangle,
\qquad \alpha_{i,j}\in\mathbb C.
\]
On the other hand, \eqref{eq:lemma-inv-positive} yields
\[
\langle S_1vG_1vS_1f_j,f_k\rangle=\delta_{jk}.
\]
Thus $\sum_jf_j\langle\cdot,f_j\rangle$, composed on the right
with $S_1vG_1vS_1$, is the identity on each basis vector $f_k$.
It is therefore the inverse on $S_1L^2$; both operators vanish
on $QL^2$. We conclude that
\begin{equation}\label{eq:lemma-inv-eigenprojection}
\begin{aligned}
S_1D_2S_1
=\sum_{j=1}^d(w\psi_j)\langle\cdot,w\psi_j\rangle
=w\mathcal P_e w, \quad
vS_1D_2S_1v=V\mathcal P_eV.
\end{aligned}
\end{equation}
The second identity uses $vw=wv=V$.

\medskip
\noindent\emph{Absolute boundedness of $D_0$.}
The operators $vG_0v$ and $vG_1v$ are absolutely bounded.
Indeed, their local singularities are integrable and can be
estimated by Young's inequality after multiplication by bounded
$v$. Away from the diagonal, the decay of $V$ gives
Hilbert--Schmidt kernels, including the growing kernel in
dimension three and the logarithmic kernel in dimension four.
Finite-rank operators on $L^2$ are Hilbert--Schmidt and hence
absolutely bounded as well.

To obtain the corresponding bound for $D_0$, choose an integer
$m>n/4$. Iteration of the Newton kernel gives
\[
\bigl|[(UvG_0v)^m](x,y)\bigr|
\lesssim v(x)v(y)
\begin{cases}
1+|x-y|^{2m-n}\mathbf1_{\{|x-y|<1\}},&2m<n,\\
1+|\log|x-y||\mathbf1_{\{|x-y|<1\}},&2m=n,\\
1,&2m>n.
\end{cases}
\]
To see this, split each Newton kernel into its part on
$|x-y|<1$ and its bounded complement. Convolution of the local
parts adds their orders, with a logarithm at order $n$.
Terms containing a bounded part are controlled using
$|V|\in L^1\cap L^\infty$.
Because $4m>n$, the displayed majorant gives
\[
\bigl\|[(UvG_0v)^m](x,\cdot)\bigr\|_2\lesssim v(x).
\]
For $f\in S_1L^2$, the identity $f=(-UvG_0v)^mf$ therefore implies
$|f(x)|\lesssim v(x)\|f\|_2$. Choosing an orthonormal basis of
$S_1L^2$, we obtain
\[
|S_1(x,y)|\lesssim v(x)v(y).
\]

For this argument, set
\[
K_0=vG_0v+S_1,\qquad J=(T+S_1)^{-1}.
\]
The preceding estimates also give
$\|[(UK_0)^m](x,\cdot)\|_2\lesssim v(x)$, and
$((UK_0)^m)^*=(K_0U)^m$.
Iterating the resolvent identity and using
$UK_0J=JK_0U=U-J$ yields the exact formula
\[
J=\sum_{j=0}^{2m-1}(-UK_0)^jU
+(UK_0)^mJ(K_0U)^m.
\]
Every term in the finite sum is absolutely bounded. By
Cauchy--Schwarz, the kernel of the last term is bounded by
\[
C\|J\|_{\mathbb B(L^2)}v(x)v(y).
\]
Thus $J$ is absolutely bounded, and so is $D_0=J-S_1$.
As usual, bounded multiplication operators such as $U$ and the
identity are included in this convention through their diagonal
measure kernels.

\medskip
\noindent\emph{Normalization and inversion of the block matrix.}
The map $f\mapsto(Qf,S_1f)$ is unitary from $L^2$ onto
$QL^2\oplus S_1L^2$. For $\lambda>0$, introduce the scaling
\[
\sigma_\lambda f=
\begin{pmatrix}Qf\\\lambda^{-1}S_1f\end{pmatrix},
\qquad
\sigma_\lambda^*(f_0,f_1)=f_0+\lambda^{-1}f_1.
\]
Here the adjoint is taken with respect to the usual direct-sum
inner product. The map $\sigma_\lambda$ is invertible. It becomes unitary when
the target is equipped with the weighted norm
$\|(f_0,f_1)\|_\lambda^2=\|f_0\|_2^2+\lambda^2\|f_1\|_2^2$.
All matrix adjoints below use the unweighted inner product.

By \eqref{eq:exp of M} and \eqref{eq:lemma-inv-cancellation},
left and right multiplication give
\begin{equation}\label{eq:lemma-inv-blocks}
\begin{aligned}
\sigma_\lambda M^+(\lambda)\sigma_\lambda^*
&=\begin{pmatrix}
QM^+Q&\lambda^{-1}QM^+S_1\\
\lambda^{-1}S_1M^+Q&\lambda^{-2}S_1M^+S_1
\end{pmatrix}\\[1mm]
&=\begin{pmatrix}QTQ&0\\0&S_1vG_1vS_1\end{pmatrix}
+E(\lambda),
\end{aligned}
\end{equation}
Here, suppressing the argument $\lambda$ in the entries,
\[
E(\lambda)=\begin{pmatrix}
E_{00}&\lambda QvG_1vS_1+E_{01}\\
\lambda S_1vG_1vQ+E_{10}&E_{11}
\end{pmatrix}
\]
and
\begin{equation}\label{eq:lemma-inv-block-orders}
\begin{aligned}
E_{00}(\lambda),\ E_{11}(\lambda)
&=\widetilde{\mathcal O}(\lambda),\\
E_{01}(\lambda),\ E_{10}(\lambda)
&=\widetilde{\mathcal O}(\lambda^2).
\end{aligned}
\end{equation}
Indeed, the rank-one terms in dimensions three and four vanish
whenever either outside projection is $S_1$.
The unscaled mixed blocks are consequently
$\lambda^2QvG_1vS_1$ and $\lambda^2S_1vG_1vQ$, modulo
$\widetilde{\mathcal O}(\lambda^3)$, and the lower diagonal block is
$\lambda^2S_1vG_1vS_1+\widetilde{\mathcal O}(\lambda^3)$.
For the upper diagonal block, use
$\lambda^2\log\lambda=\mathcal O_{[(n+3)/2]}(\lambda)$.
Thus $E(\lambda)=\widetilde{\mathcal O}(\lambda)$, with the
sharper mixed-entry remainders in \eqref{eq:lemma-inv-block-orders}.

The constant matrix in \eqref{eq:lemma-inv-blocks} is invertible.
Write
\[
\mathcal D=\begin{pmatrix}D_0&0\\0&D_2\end{pmatrix},
\qquad
\mathcal B(\lambda)
=\bigl(\sigma_\lambda M^+(\lambda)\sigma_\lambda^*\bigr)^{-1}.
\]
For sufficiently small $\lambda_0>0$ and $0<\lambda<\lambda_0$,
the Neumann expansion gives
\begin{equation}\label{eq:lemma-inv-normalized-inverse}
\begin{aligned}
\mathcal B(\lambda)
&=(I+\mathcal D E(\lambda))^{-1}\mathcal D\\
&=\mathcal D-\mathcal D E(\lambda)\mathcal D
+\mathcal D E(\lambda)\mathcal D E(\lambda)\mathcal B(\lambda),
\end{aligned}
\end{equation}
with $\mathcal B(\lambda)=\widetilde{\mathcal O}(1)$.
To verify the absolute-kernel estimates, choose nonnegative
bounded kernel majorants for $\mathcal D$ and for
$\lambda^{k-1}\partial_\lambda^kE$, $0\le k\le[(n+3)/2]$.
The $\ell$-th Neumann term is then dominated by the corresponding
product of majorants times $\lambda^\ell$.
After $k$ derivatives and multiplication by $\lambda^k$, at most
a polynomial factor in $\ell$ is introduced. Taking $\lambda_0$
small makes these majorant series converge in operator norm.
Their sums supply the required common nonnegative kernels.
In particular, the last term in
\eqref{eq:lemma-inv-normalized-inverse} is
$\widetilde{\mathcal O}(\lambda^2)$ in every entry.

Keeping the first-order mixed terms explicit, we obtain
\begin{equation}\label{eq:lemma-inv-matrix-expansion}
\mathcal B(\lambda)=
\begin{pmatrix}
D_0+\widetilde{\mathcal O}(\lambda)
&-\lambda D_0vG_1vD_2+\widetilde{\mathcal O}(\lambda^2)\\[1mm]
-\lambda D_2vG_1vD_0+\widetilde{\mathcal O}(\lambda^2)
&D_2+S_1G(\lambda)S_1
\end{pmatrix},
\end{equation}
where $G(\lambda)=\widetilde{\mathcal O}(\lambda)$ is defined
as the exact lower diagonal entry of $\mathcal B(\lambda)$ minus
$D_2$, extended by zero on $QL^2$.
Thus all higher-order terms in this entry are retained in $G$.
The other three remainders retain their respective outside projections.

\medskip
\noindent\emph{Restoring the scaling.}
Since $\sigma_\lambda$ is invertible, the exact identity is
\begin{equation}\label{eq:lemma-inv-scale-back}
\begin{aligned}
(M^+(\lambda))^{-1}
=\sigma_\lambda^*\mathcal B(\lambda)\sigma_\lambda
=\begin{pmatrix}Q&\lambda^{-1}S_1\end{pmatrix}
\mathcal B(\lambda)
\begin{pmatrix}Q\\\lambda^{-1}S_1\end{pmatrix}.
\end{aligned}
\end{equation}
The upper diagonal entry contributes
$D_0+\widetilde{\mathcal O}(\lambda)$.
Each mixed entry is multiplied by $\lambda^{-1}$, giving
\[
-D_0vG_1vD_2+\widetilde{\mathcal O}(\lambda),
\qquad
-D_2vG_1vD_0+\widetilde{\mathcal O}(\lambda),
\]
respectively. The lower diagonal entry is multiplied by
$\lambda^{-2}$. Consequently,
\[
(M^+(\lambda))^{-1}
=\frac{S_1(D_2+G(\lambda))S_1}{\lambda^2}
+\mathcal N+\mathcal R(\lambda),
\]
where
\begin{equation}\label{eq:lemma-inv-N}
\begin{aligned}
\mathcal N=D_0-D_0vG_1vD_2-D_2vG_1vD_0,\quad
\mathcal R(\lambda)=\widetilde{\mathcal O}(\lambda).
\end{aligned}
\end{equation}
The operator $\mathcal N$ is absolutely bounded by the estimates
above. Moreover, $S_1\mathcal R(\lambda)S_1=0$, since the entire
lower diagonal correction has been incorporated into $G$.
The $\widetilde{\mathcal O}(\lambda^2)$ bounds on the mixed
remainders are precisely what makes their restored contributions
$\widetilde{\mathcal O}(\lambda)$. This proves the lemma.
\end{proof}
\vskip0.5cm
\subsection*{Acknowledgements}    
This Project was supported by the National Key R\&D program of China: No. 2021YFA1002500.
    H. Cheng was supported by CPSF under grant 2025M784444. C. Miao was supported by NSFC under grants 12371095 and 12531005.
   X. Yao was supported by NSFC (Grants No.~12531005  and 12671120).

\end{document}